\documentclass[a4paper]{article}
\usepackage{hyperref}
\usepackage{graphicx}
\usepackage{mathrsfs}
\usepackage{enumerate} 
\usepackage{amsxtra,amssymb,latexsym, amscd,amsthm}
\usepackage{indentfirst}
\usepackage{color}
\usepackage[utf8]{inputenc}
\usepackage[mathscr]{eucal}
\usepackage{amsfonts}
\usepackage{graphics}
\usepackage{multirow}
\usepackage{array}
\usepackage{subfigure}
\usepackage{cite}
\usepackage{wrapfig}
\usepackage{tikz}
\usepackage{diagbox}

\usepackage{pgfplots}
\pgfplotsset{compat=1.15}
\usepackage{mathrsfs}
\usetikzlibrary{arrows}

\newcommand{\footremember}[2]{%
    \footnote{#2}
    \newcounter{#1}
    \setcounter{#1}{\value{footnote}}%
}
 
\def\R{{\mathbb R}}

\def\N{{\mathbb N}}

\DeclareMathOperator{\rank}{rank}
\DeclareMathOperator{\diag}{diag}
\DeclareMathOperator{\FJ}{FJ}
\DeclareMathOperator{\KKT}{KKT}
\DeclareMathOperator{\bit}{bit}

\newtheorem{theorem}{\bf Theorem}
\newtheorem{lemma}{\bf Lemma}
\newtheorem{remark}{\bf Remark}

\providecommand{\keywords}[1]
{
  \small	
  \textbf{\textbf{Keywords:}} #1
}
\begin{document}
\definecolor{qqzzff}{rgb}{0,0.6,1}
\definecolor{ududff}{rgb}{0.30196078431372547,0.30196078431372547,1}
\definecolor{xdxdff}{rgb}{0.49019607843137253,0.49019607843137253,1}
\definecolor{ffzzqq}{rgb}{1,0.6,0}
\definecolor{qqzzqq}{rgb}{0,0.6,0}
\definecolor{ffqqqq}{rgb}{1,0,0}
\definecolor{uuuuuu}{rgb}{0.26666666666666666,0.26666666666666666,0.26666666666666666}
\newcommand{\vi}[1]{\textcolor{blue}{#1}}
\newif\ifcomment
\commentfalse
\commenttrue
\newcommand{\comment}[3]{%
\ifcomment%
	{\color{#1}\bfseries\sffamily#3%
	}%
	\marginpar{\textcolor{#1}{\hspace{3em}\bfseries\sffamily #2}}%
	\else%
	\fi%
}

\newcommand{\mapr}[1]{{{\color{blue}#1}}}
\newcommand{\revise}[1]{{{\color{blue}#1}}}

\title{Complexity for exact polynomial optimization strengthened with Fritz John conditions}


\author{%
Ngoc Hoang Anh Mai\footremember{1}{CNRS; LAAS; 7 avenue du Colonel Roche, F-31400 Toulouse; France.}
  }

\maketitle

\begin{abstract}
Let $f,g_1,\dots,g_m$ be polynomials of degree at most $d$  with real coefficients in a vector of variables $x=(x_1,\dots,x_n)$.
Assume that $f$ is non-negative on a basic semi-algebraic set $S$ defined by polynomial inequalities $g_j(x)\ge 0$, for $j=1,\dots,m$.
Our previous work [arXiv:2205.04254 (2022)] has stated several representations of $f$ based on the Fritz John conditions.
This paper provides some explicit degree bounds depending on $n$, $m$, and $d$ for these representations.
In application to polynomial optimization, we obtain explicit rates of finite convergence of the hierarchies of semidefinite relaxations based on these representations.
\end{abstract}
\tableofcontents
\section{Introduction}
\paragraph{Nichtnegativstellens\"atze and their degree bounds.} Hilbert's 17th problem concerns the representation of a non-negative polynomial as a sum of squares of rational functions.
Artin gave a positive answer to this problem in \cite{artin1927zerlegung}.
Later Krivine \cite{krivine1964anneaux} and Stengle \cite{stengle1974nullstellensatz} stated a Nichtnegativstellensatz.
It says that we can express a polynomial $f$ non-negative on a basic semi-algebraic set $S$ as a linear combination of the polynomials $g_j$ defining $S$ with weights which are sums of squares of rational functions.
Lombardi, Perrucci, and Roy have recently provided a degree bound for Krivine--Stengle's Nichtnegativstellensatz.
It is a tower of five exponentials depending only on the number of variables, the number of polynomials $g_j$ defining $S$, and the degrees of $f,g_j$.
Consequently, they obtain the best-known bound degrees for the numerators and denominators in Hilbert--Artin's Nichtnegativstellensatz as a tower of five exponentials.

\paragraph{Positivstellens\"atze.} 
Today Positivstellens\"atze, representations of polynomials positive on a basic semi-algebraic set,  are of broad interest with influential applications in polynomial optimization.
Let us revisit two representations commonly used in practice:
Schm\"udgen's Positivstellensatz \cite{schmudgen1991thek} says that we can write a polynomial $f$ positive on a compact basic semi-algebraic set $S$ as a linear combination of products of polynomials defining $S$ with weights which are sums of squares of polynomials.
Putinar presents in \cite{putinar1993positive} a Positivstellensatz saying that we con decompose a polynomial $f$ positive on a compact basic semi-algebraic set $S$ satisfying the so-called Archimedean condition (stated in  \cite{mai2022exact}) as a linear combination of polynomials defining $S$ with weights which are sums of squares of polynomials.

\paragraph{Polynomial optimization.} 
Let $\R[x]$ denote the ring of polynomials with real coefficients in the vector of variables $x$.
Given $r\in\N$, denote by $\R_r[x]$ the linear space of polynomials in $\R[x]$ of degree at most $r$.
Denote by $\Sigma^2[x]$ (resp. $\Sigma^2_r[x]$) the cone of sum of squares of polynomials in $\R[x]$ (resp. $\R_r[x]$).

Given $f,g_1,\dots,g_m\in\R[x]$, consider polynomial optimization problem:
\begin{equation}\label{eq:pop}
    f^\star:=\inf\limits_{x\in S(g)} f(x)\,,
\end{equation}
where $S(g)$ is the basic semi-algebraic set associated with $g=(g_1,\dots,g_m)$, i.e.,
\begin{equation}
    S(g):=\{x\in\R^n\,:\,g_j(x)\ge 0\,,\,j=1,\dots,m\}\,.
\end{equation}
Schm\"udgen's and Putinar's Positivstellensatz are both applicable to polynomial optimization.
Indeed, the idea is to reformulate the original problem as
$f^\star=\sup\{\lambda\in\R\,:\,f-\lambda> 0\text{ on }S(g)\}$
and then replace the constraint ``$f-\lambda> 0\text{ on }S(g)$" with the representations of $f-\lambda$ associated with $g_j$.
More explicitly, using Putinar's Positivstellensatz gives a sequence of reals $(\rho_k^\text{Pu})_{k\in\N}$, where
\begin{equation}
\begin{array}{rl}
\rho_k^\text{Pu}:=\sup\limits_{\lambda,\sigma_j}&\lambda\\
&\lambda\in\R\,,\,\sigma_j\in\Sigma^2[x]\,,\\
&f-\lambda=\sum_{j=1}^m\sigma_jg_j\,,\,\deg(\sigma_jg_j)\le 2k\,,
\end{array}
\end{equation}
where $g_1=1$. 
Here $\deg(\cdot)$ stands for the degree of a polynomial.
The sequence $(\rho_k^\text{Pu})_{k\in\N}$ converges to $f^\star$ under the Archimedean condition on $S$.
Moreover, the program to compute each $\rho_k^\text{Pu}$ is convex, specifically, semidefinite (see \cite{boyd2004convex}).
The sequence of these programs is known as Lasserre's hierarchy \cite{lasserre2001global}.
A similar process applies to Schm\"udgen's Positivstellensatz to obtain the corresponding sequence of reals $(\rho_k^\text{Sch})_{k\in\N}$ defined by 
\begin{equation}
\begin{array}{rl}
\rho_k^\text{Sch}:=\sup\limits_{\lambda,\sigma_j}&\lambda\\
&\lambda\in\R\,,\,\sigma_j\in\Sigma^2[x]\,,\\
&f-\lambda=\sum_{\alpha\in\{0,1\}^m}\sigma_\alpha g_1^{\alpha_1}\dots g_m^{\alpha_m}\,,\,\deg(\sigma_\alpha g_\alpha)\le 2k\,.
\end{array}
\end{equation}

\paragraph{Convergence rates.}  
Schweighofer \cite{schweighofer2004complexity} (resp. Nie--Schweighofer \cite{nie2007complexity}) analyzes the convergence rates of sequence $(\rho_k^\text{Sch})_{k\in\N}$ (resp. $(\rho_k^\text{Pu})_{k\in\N}$).
Despite the polynomial time complexity $\mathcal{O}(\varepsilon^{-c})$ of the former, the latter has unexpected exponential time complexity $\mathcal{O}(\exp(\varepsilon^{-c}))$. 
Baldi and Mourrain have recently provided in \cite{baldi2021moment} an improved complexity $\mathcal{O}(\varepsilon^{-c})$ for sequence $(\rho_k^\text{Pu})_{k\in\N}$.
It relies on the degree bound for Schmudgen's Positivstellensatz on the unit hypercube stated by Laurent and Slot in \cite{laurent2021effective}.
In addition, Laurent and Slot obtain the optimal convergence rate $\mathcal{O}(\varepsilon^{-1/2})$ for minimizing a polynomial on the unit hypercube.
To do this, Laurent and Slot utilize the polynomial kernel method introduced in \cite{fang2020sum}, where Fang and Fawzi provide the optimal convergence rate $\mathcal{O}(\varepsilon^{-1/2})$ for minimizing a polynomial on the unit sphere.
Applying this method again, Slot achieves in \cite{slot2021sum} the optimal convergence rate for $(\rho_k^\text{Sch})_{k\in\N}$ when $S$ is the unit ball or the standard simplex.

\paragraph{Finite convergence.} The finite convergences of the two sequences $(\rho_k^\text{Sch})_{k\in\N}$ and $(\rho_k^\text{Pu})_{k\in\N}$ are mainly studied by Scheiderer \cite{scheiderer2000sums,scheiderer2003sums,scheiderer2006sums}, 
Marshall \cite{marshall2006representations,marshall2009representations}
and Nie \cite{nie2014optimality}.
In particular, Lasserre's hierarchy has finite convergence under the Archimedean condition and some standard optimality conditions. 
These conditions make the original polynomial optimization problem necessarily have finitely many global minimizers at which the Karush--Kuhn--Tucker conditions hold for this problem.
Recent efforts of Nie--Demmel--Sturmfels \cite{nie2006minimizing}, Demmel--Nie--Powers \cite{demmel2007representations}, Nie \cite{nie2013exact}, and our previous work \cite{mai2022exact} are to obtain the finite convergence of Lasserre's hierarchy through the information of the Jacobian of the objective and constrained polynomials.
However, rates of finite convergence for these methods have been open until now.

\paragraph{Nichtnegativstellens\"atze based on Fritz John conditions.}
Let $V(h_{\FJ})$ and $V(h_{\FJ}^+)$ (defined later in \eqref{eq:.polyFJ} and \eqref{eq:.polyFJ.plus}) be the sets of points at which the Fritz John conditions hold for problem \eqref{eq:pop}.
In our previous works \cite{mai2022exact}, we use the finiteness of the images of $S(g) \cap V(h_{\FJ})$ and $S(g) \cap V(h_{\FJ}^+)$ under $f$ to construct representations of $f-f^\star$ without denominators involving quadratic modules and preorderings associated with these two intersections. 
Here $f^\star$ is defined as in \eqref{eq:pop}.
These non-negativity certificates allow us to tackle the following two cases: (i) polynomial $f-f^\star$ is non-negative with infinitely many zeros on basic semi-algebraic sets $S(g)$ and (ii) the Karush–Kuhn–Tucker conditions do not hold for problem \eqref{eq:pop} at any zero of $f-f^\star$ on $S(g)$.

\paragraph{Contribution.} 
In this paper, we aim to provide some explicit degree bounds for Nichtnegativstellens\"atze in our previous work \cite{mai2022exact} and the work of Demmel, Nie, and Powers \cite{demmel2007representations}.
To do this, we rely on the constructiveness of these representations and utilize the degree bounds for Krivine--Stengle's Nichtnegativstellens\"atze analyzed by Lombardi, Perrucci, and Roy \cite{lombardi2020elementary} as well as the  upper bound on the number of connected
components of a basic semi-algebraic set due to Coste \cite{coste2000introduction}.

We briefly sketch the analysis for the representation of $f-f^\star$ using the set $V(h_{\FJ})$ as follows:
We first decompose $V(h_{\FJ})$ into finitely many connected components.
Thanks to Coste \cite{coste2000introduction}, the number of these components is bounded from above by a value depending on the number of polynomials defining $V(h_{\FJ})$ and the degrees of these polynomials.
To prove that $f$ is constant on each connected component, we rely on the Fritz John conditions generating $V(h_{\FJ})$ and the assumption that the image of the set of critical points $C(g)$ (defined later in \eqref{eq:critical.set}) under $f$ is finite.
It turns out that $f$ has finitely many values on $V(h_{\FJ})$. 
To obtain the representation of $f-f^\star$ with explicit degree bound, we explicitly construct a ``variety version" of Lagrangian interpolation and utilize  Krivine--Stengle's Nichtnegativstellens\"atze with degree bound given in \cite{lombardi2020elementary}.

We accordingly obtain the explicit convergence rate of each corresponding hierarchy of semidefinite relaxations for a polynomial optimization problem.
The rates of these hierarchies depend on the number of variables, the number of constraints, and the degrees of the input polynomials.
In particular, they do not depend on the coefficients of the input polynomials.
Because of the huge bounds on Krivine--Stengle's Nichtnegativstellens\"atze and the number of connected
components of a basic semi-algebraic set, the rates   are extremely large. 

\paragraph{Organization.}
We organize the paper as follows:
Section \ref{sec:preli} is to recall some necessary tools from real algebraic geometry.  
Section \ref{sec:degree.bound} states the explicit degree bounds for our Nichtnegativstellens\"atze based on the Fritz John conditions. 
Section \ref{sec:application} presents the application of these bounds in analyzing the complexity of the corresponding hierarchies of the exact semidefinite relaxations for a polynomial optimization problem. 
Section \ref{sec:bound.KKT} provides the explicit degree bounds for the representation based on the Karush--Kuhn--Tucker conditions by Demmel, Nie, and Powers in \cite{demmel2007representations}.
Section \ref{sec:implicit.quadraic} shows the implicit degree bound for the representations involving quadratic modules.

\section{Preliminaries}
\label{sec:preli}
This section presents some preliminaries from real algebraic geometry needed to prove our main results.

\subsection{First-order optimality conditions}

Given $p\in\R[x]$, we denote by $\nabla p$ the gradient of $p$, i.e., $\nabla p=(\frac{\partial p}{\partial x_1},\dots,\frac{\partial p}{\partial x_n})$.
We say that the Fritz John conditions hold for problem \eqref{eq:pop} at $u\in S(g)$  if 
\begin{equation}\label{eq:FJcond}
\begin{cases}
    		\exists (\lambda_0,\dots,\lambda_m)\in[0,\infty)^{m+1}\,:\\
          \lambda_0 \nabla f(u)=\sum_{j=1}^m \lambda_j \nabla g_j(u)\,,\\
          \lambda_j g_j(u) =0\,,\,j=1,\dots,m\,,\\
          \sum_{j=0}^m \lambda_j^2=1\,,
    \end{cases}
   \Leftrightarrow
    \begin{cases}
         \exists (\lambda_0,\dots,\lambda_m)\in\R^{m+1}\,:\\
         \lambda_0^2 \nabla f(u)=\sum_{j=1}^m \lambda_j^2 \nabla g_j(u)\,,\\
          \lambda_j^2 g_j(u) =0\,,\,j=1,\dots,m\,,\\
          \sum_{j=0}^m \lambda_j^2=1\,.
    \end{cases}
\end{equation}
In addition, the Karush--Kuhn--Tucker conditions hold for problem \eqref{eq:pop} at $u\in S(g)$  if 
\begin{equation}\label{eq:KKT.cond}
\begin{cases}
    		\exists (\lambda_1,\dots,\lambda_m)\in[0,\infty)^{m}\,:\\
          \nabla f(u)=\sum_{j=1}^m \lambda_j \nabla g_j(u)\,,\\
          \lambda_j g_j(u) =0\,,\,j=1,\dots,m\,.
    \end{cases}
   \Leftrightarrow
\begin{cases}
    		\exists (\lambda_1,\dots,\lambda_m)\in\R^{m}\,:\\
          \nabla f(u)=\sum_{j=1}^m \lambda_j^2 \nabla g_j(u)\,,\\
          \lambda_j^2 g_j(u) =0\,,\,j=1,\dots,m\,.
    \end{cases}
\end{equation}

If $u$ is a  local minimizer for problem \eqref{eq:pop}, then the Fritz John conditions hold for problem \eqref{eq:pop} at $u$.
In contrast, there exist cases where the Karush--Kuhn--Tucker conditions do not hold for problem \eqref{eq:pop} at any local minimizer of this problem.
We denote by $W(f,g)$ the set of all points at which the Fritz John conditions hold but the Karush--Kuhn--Tucker conditions do not hold for problem \eqref{eq:pop}.

\subsection{Sets of critical points}
Given $g_1,\dots,g_m\in\R[x]$, let  $\varphi^g:\R^{n}\to \R^{(n+m)\times m}$ be a function associated with $g=(g_1,\dots,g_m)$ defined by
\begin{equation}
\varphi^g(x)=\begin{bmatrix}
\nabla g(x)\\
\diag(g(x))
\end{bmatrix}=
    \begin{bmatrix}
    \nabla g_1(x)& \dots& \nabla g_m(x)\\
    g_1(x)&\dots&0\\
    .&\dots&.\\
    0&\dots&g_m(x)
    \end{bmatrix}\,.
\end{equation}
Given a real matrix $A$, we denote by $\rank(A)$ the dimension of the vector space generated by the columns of $A$ over $\R$.
We say that a set $\Omega$ is finite if its cardinal number is a non-negative integer.
Let  $C(g)$ be the set of critical points associated with $g$ defined by
\begin{equation}\label{eq:critical.set}
    C(g):=\{x\in\R^n\,:\,\rank(\varphi^g(x))< m\}.
\end{equation}
Given a real matrix $A$, we denote by $\rank^+(A)$ the largest number of columns of $A$ whose convex hull over $\R$ has no zero.
Let  $C^+(g)$ be the set of critical points associated with $g$ defined by
\begin{equation}
    C^+(g):=\{x\in \R^n\,:\,\rank^+(\varphi^g(x))< m\}.
\end{equation}
\subsection{Quadratic modules, preoderings, and ideals}
Given $g_1,\dots,g_m\in\R[x]$, let $Q_r(g)[x]$ be the truncated quadratic module of order $r\in\N$ associated with $g=(g_1,\dots,g_m)$, i.e., 
\begin{equation}
    Q_r(g)[x]:=\{\sigma_0 +\sum_{j=1}^m g_j\sigma_j\,:\,\sigma_j\in\Sigma^2[x]\,,\,\deg(\sigma_0)\le 2r\,,\,\deg(g_j\sigma_j)\le 2r\}\,.
\end{equation}
Let $\Pi g$ be the vector of products of $g_1,\dots,g_m$ defined by
\begin{equation}\label{eq:prod.g}
\Pi g:=(g^\alpha)_{\alpha\in\{0,1\}^m\backslash \{0\}}\,,
\end{equation}
where $\alpha=(\alpha_1,\dots,\alpha_m)$ and $g^\alpha:=g_1^{\alpha_1}\dots g_m^{\alpha_m}$.
We call $Q_r(\Pi g)[x]$ the truncated preordering of order $r\in\N$ generated by $g$, denoted by $P_r(g)[x]$.
Obviously, if $m=1$, it holds that $P_r(g)[x]=Q_r(g)[x]$.

Given $h_1,\dots,h_l\in\R[x]$, let $V(h)$ be the variety generated by $h=(h_1,\dots,h_l)$, i.e.,
\begin{equation}
V(h):=\{x\in\R^n\,:\,h_j(x)=0\,,\,j=1,\dots,l\}\,.
\end{equation}
and let $I_r(h)[x]$ be the truncated ideal of order $r$ generated by $h$, i.e.,
\begin{equation}
    I_r(h)[x]:= \{\sum_{j=1}^l h_j \psi_j\,:\,\psi_j\in\R[x]\,,\,\deg(h_j \psi_j)\le 2r\}\,.
\end{equation}

\subsection{Degree bound for Krivine--Stengle's Nichtnegativstellens\"atze}
We denote by $\bit(d)$ the number of bits of $d\in\N$, i.e.,
\begin{equation}
\bit(d):=
 \begin{cases}
          1 & \text{if } d = 0\,,\\
	    k & \text{if } d \ne 0 \text{ and } 2^{k-1}\le d < 2^k.
         \end{cases}
\end{equation}
Given $n,d,s\in\N$, set
\begin{equation}
b(n,d,s):=2^{
2^{\left(2^{\max\{2,d\}^{4^{n}}}+s^{2^{n}}\max\{2, d\}^{16^{n}\bit(d)} \right)}}\,.
\end{equation}

We recall the degree bound for Krivine--Stengle's Nichtnegativstellens\"atze analyzed by Lombardi, Perrucci, and Roy \cite{lombardi2020elementary} in the following two lemmas:
\begin{lemma}\label{lem:pos}
Let $g_1,\dots,g_m,h_1,\dots,h_l\in\R_d[x]$. 
Assume that $S(g)\cap V(h)=\emptyset$ with $g:=(g_1,\dots,g_m)$ and $h:=(h_1,\dots,h_l)$. 
Set $r=b(n,d,m+l+1)/2$.
Then it holds that $-1 \in P_r(g)[x]+I_r(h)[x]$.
\end{lemma}
\begin{lemma}\label{lem:pos2}
Let $p,g_1,\dots,g_m,h_1,\dots,h_l\in\R_d[x]$. 
Assume that $p$ vanishes on $S(g)\cap V(h)$ with $g:=(g_1,\dots,g_m)$ and $h:=(h_1,\dots,h_l)$. 
Set $r:=b(n,d,m+l+1)/2$ and $s:=2\lfloor r/d\rfloor$.
Then it holds that $-p^s \in P_r(g)[x]+ I_r(h)[x]$.
\end{lemma}

\subsection{Upper bound on the number of connected components of a basic semi-algebraic set}
We generalize the definition of basic semi-algebraic sets as follows: A semi-algebraic subset of $\R^n$ is a subset of the form
\begin{equation}\label{eq:def.semi.set}
\bigcup_{i=1}^t\bigcap_{j=1}^{r_i}\{x\in\R^n\,:\,f_{ij}(x)*_{ij}0\}\,,
\end{equation}
where $f_{ij}\in\R[x]$ and $*_{ij}$ is either $>$ or $=$.

Given two semi-algebraic sets $A\subset \R^n$ and $B\subset \R^m$, we say that a mapping $f : A \to B$ is semi-algebraic if its graph $\{(x,f(x))\,:\,x\in A\}$ is a semi-algebraic set in $\R^{n+m}$.
A semi-algebraic subset $A\subset \R^n$ is said to be semi-algebraically path connected if for every $x,y$ in $A$, there exists a continuous semi-algebraic mapping $\phi:[0,1] \to A$ such that $\phi(0) = x$ and $\phi(1) = y$.
Note that $\phi$ is piecewise-differentiable in this case (see, e.g., \cite[Theorem 1.8.1]{pham2016genericity}).

Given $n,d,s\in\N$, set
\begin{equation}
c(n,d,s):=d(2d-1)^{n+s-1}\,.
\end{equation}

The upper bound on the number of connected components of a basic semi-algebraic set is stated by Coste \cite[Proposition 4.13]{coste2000introduction} as follows:
\begin{lemma}\label{lem:num.connected}
Let $g_1,\dots,g_m,h_1,\dots,h_l\in\R_d[x]$ with $d\ge 2$. 
The number of (semi-algebraically path) connected components of $S(g)\cap V(h)$ is not greater than $c(n,d,m+l)$.
\end{lemma}

\subsection{Sum-of-squares representations under the finite image assumption}
Let $|\cdot|$ stand for the cardinal number of a set.
Denote by $\delta_{ij}$ the Kronecker delta function at $(i,j)\in\N^2$.

We state in the following lemmas the representations of polynomials nonnegative on semi-algebraic sets under the finite image assumption:
\begin{lemma}\label{lem:quadra}
Let $f,g_1,\dots,g_m\in\R_d[x]$ and $h_1,\dots,h_l\in\R_{d+1}[x]$. 
Assume that $f$ is non-negative on $S(g)$ with $g=(g_1,\dots,g_m)$ and $f(V(h))$ is finite with $h=(h_1,\dots,h_l)$. 
Set $r:=|f(V(h))|$ and $u:=b(n,d+1,m+l+2)/2$.
Then there exist $q\in P_w(g)[x]$ with $w=dr+u$ such that $f - q$ vanishes on $V(h)$.
\end{lemma}
\begin{proof}
By assumption, we get
$f(V(h)) = \{t_1 ,\dots, t_r \} \subset \R$,
where $t_i\ne t_j$ if $i\ne j$.
For $j=1,\dots,r$, let
$W_j:=V(h,f-t_j)$.
Then $W_j$ is a real variety generated by $l+1$ polynomials in $\R_{d+1}[x]$.
It is clear that $f(W_j)=\{t_j\}$.

Define the following polynomials:
\begin{equation}\label{eq:lagrange.pol}
p_j(x)=\prod_{i\ne j}\frac{f(x)-t_i}{t_j-t_i}\,,\,j=1,\dots,r\,.
\end{equation}
It is easy to check that $p_j(W_i)=\{\delta_{ji}\}$ and $\deg(p_j)\le d(r-1)$.
Without loss of generality, we assume that there is $s\in \{0,1,\dots,r-1\}$ such that $W_j\cap S(g)= \emptyset$, for $j=1,\dots,s$, and $W_i\cap S(g)\ne \emptyset$, for $i=s+1,\dots,r$.

Let $j\in\{1,\dots,s\}$.
Since $W_j \cap S(g) = \emptyset$, Lemma \ref{lem:pos} says that $-1 \in P_u(g)[x]+I_u(h,f-t_j)[x]$.
It implies that there exists $v_j \in P_u(g)[x]$ such that
$-1 = v_j$ on $W_j$. 
We have $f = s_1 - s_2$ for the SOS polynomials $s_1 =(f+\frac{1}{2})^2$ and $s_2 = f^2+\frac{1}{4}$.
It implies that $f = s_1 + v_j s_2$ on $W_j$.
Let $q_j = s_1 +v_j s_2 \in P_{u+d}(g)[x]$.

Since $f\ge 0$ on $S(g)$, it holds that $f=t_i\ge 0$ on $W_i$, for $i=s+1,\dots,r$.  
Now letting
\begin{equation}\label{eq:rep.q}
q =\sum_{j=1}^s q_j p_j^2+\sum_{i=s+1}^r t_i p_i^2\,,
\end{equation}
we obtain $q\in P_w(g)[x]$ since $\deg(t_i p_i^2)\le 2\deg(p_i)\le 2d(r-1)\le 2w$ and
\begin{equation}
\deg(q_j p_j^2)\le \deg(q_j)+2\deg(p_j)\le 2(d+u+d(r-1))=2w\,,
\end{equation}
Hence $f - q$ vanishes on $V(h)=W_1\cup\dots\cup W_r$.
\end{proof}
\begin{lemma}\label{lem:quadra.deno}
Let $f,g_1,\dots,g_m\in\R_d[x]$ and $h_1,\dots,h_l\in\R_{d+1}[x]$. 
Assume that $f$ is non-negative on $S(g)$ with $g=(g_1,\dots,g_m)$ and $f(V(h)\backslash A)$ is finite with $h=(h_1,\dots,h_l)$ and $A\subset \R^n$. 
Set $r:=|f(V(h)\backslash A)|$ and $u:=b(n,d+1,m+l+2)/2$.
Then there exist $q\in P_w(g)[x,\bar\lambda]$ with $w=dr+u$ such that $f - q$ vanishes on $V(h)\backslash A$.
\end{lemma}
\begin{proof}
By assumption, we get
$f(V(h)\backslash A) = \{t_1 ,\dots, t_r \} \subset \R$,
where $t_i\ne t_j$ if $i\ne j$.
We now handle in much the same way as the proof of Lemma \ref{lem:quadra} to get $q\in P_w(g)[x]$  in \eqref{eq:rep.q}.
Hence $f - q$ vanishes on $V(h)\backslash A\subset W_1\cup\dots\cup W_r$.
\end{proof}

\begin{lemma}\label{lem:quadra2}
Let $f,g_1,\dots,g_m\in\R_d[x]$ and $h_1,\dots,h_l\in\R_{d+1}[x]$. 
Assume that $f$ is non-negative on $S(g)$ with $g=(g_1,\dots,g_m)$ and $f(S(g)\cap V(h))$ is finite with $h=(h_1,\dots,h_l)$. 
Set $r:=|f(S(g)\cap V(h))|$.
Then there exist $q\in P_w(g)[x,\bar\lambda]$ with $w=d(r-1)$ such that $f - q$ vanishes on $S(g)\cap V(h)$.
\end{lemma}
\begin{proof}
By assumption, we get
$f(S(g)\cap V(h)) = \{t_1 ,\dots, t_r \} \subset \R$,
where $t_i\ne t_j$ if $i\ne j$.
For $j=1,\dots,r$, let
$W_j:=V(h,f-t_j)$.
Then $W_j$ is a real variety generated by $l+1$ polynomials in $\R_{d+1}[x]$.
It is clear that $f(W_j)=\{t_j\}$.
Define polynomials $p_j$, for $j=1,\dots,r$, as in \eqref{eq:lagrange.pol}.
It is easy to check that $W_i\cap S(g)\ne \emptyset$, $p_j(W_i)=\{\delta_{ji}\}$ and $\deg(p_j)\le d(r-1)$.
Since $f\ge 0$ on $S(g)$, it holds that $f=t_i\ge 0$ on $W_i$, for $i=1,\dots,r$.  
Now letting
$q =\sum_{i=1}^r t_i p_i^2$,
we obtain $q\in \Sigma_w^2[x]$ with $w=d(r-1)$, and hence $f - q$ vanishes on $W_1\cup\dots\cup W_r\supset S(g)\cap V(h)$.
\end{proof}

\subsection{Moment/Sum-of-squares relaxations for polynomial optimization}

We recall some preliminaries  of the Moment/Sum-of-squares relaxations originally developed by Lasserre in \cite{lasserre2001global}.

Given $d\in\N$, let $\N^n_d:=\{\alpha\in\N^n\,:\,\sum_{j=1}^n \alpha_j\le d\}$.
Given $d\in\N$, we denote by $v_d$ the vector of monomials in $x$ of degree at most $d$, i.e., $v_d=(x^\alpha)_{\alpha\in\N^n_d}$ with $x^\alpha:=x_1^{\alpha_1}\dots x_n^{\alpha_n}$.
For each $p\in\R_d[x]$, we write $p=c(p)^\top v_d=\sum_{\alpha\in\N^n_d}p_\alpha x^\alpha$, where $c(p)$ is denoted by the vector of coefficient of $p$, i.e., $c(p)=(p_\alpha)_{\alpha\in\N^n_d}$ with $p_\alpha\in\R$.
Given $A\in\R^{r\times r}$ being symmetric, we say that $A$ is positive semidefinite, denoted by $A\succeq 0$, if every eigenvalue of $A$ is non-negative.

\paragraph{Moment/Localizing matrices.} Given $y=(y_\alpha)_{\alpha\in\N^n}\subset \R$, let $L_y:\R[x]\to\R$ be the Riesz linear functional defined by $L_y(p)=\sum_{\alpha\in\N^n} p_\alpha y_\alpha$ for every $p\in\R[x]$.
Given $d\in\N$, $p\in\R[x]$ and $y=(y_\alpha)_{\alpha\in\N^n}\subset \R$, let $M_d(y)$ be the moment matrix of order $d$ defined by $(y_{\alpha+\beta})_{\alpha,\beta\in\N^n_d}$ and let $M_d(py)$ be the localizing matrix of order $d$ associated with $p$ defined by $(\sum_{\gamma\in\N^n}p_\gamma y_{\alpha+\beta+\gamma})_{\alpha,\beta\in\N^n_d}$.

\paragraph{Truncated quadratic modules/ideals.} Given $g_1,\dots,g_m\in\R[x]$, let $Q_d(g)[x]$ be the truncated quadratic module of order $d$ associated with $g=(g_1,\dots,g_m)$ defined by
\begin{equation}
Q_d(g)[x]=\{\sigma_0+\sum_{j=1}^m\sigma_jg_j\,:\,\sigma_j\in\Sigma^2[x]\,,\,\deg(\sigma_0)\le 2d\,,\,\deg(\sigma_jg_j)\le 2d\}\,.
\end{equation}
Given $h_1,\dots,h_l\in\R[x]$, let $I_d(h)$ be the truncated ideal of order $d$ associated with $h=(h_1,\dots,h_l)$ defined by
\begin{equation}
I_d(h)[x]=\{\sum_{j=1}^l\psi_jh_j\,:\,\psi_j\in\R[x]\,,\,\deg(\psi_jh_j)\le 2d\}\,.
\end{equation}
\paragraph{Problem statement.} Consider polynomial optimization problem:
\begin{equation}\label{eq:pop.equality}
\bar f^\star:=\inf\limits_{x\in S(g)\cap V(h)} f(x)\,,
\end{equation}
where $g=(g_1,\dots,g_m)$ and $h=(h_1,\dots,h_l)$ with  $f,g_i,h_j\in\R[x]$.
\subsubsection{The case without denominators}
Given $k\in\N$ and $f,g_1,\dots,g_m,h_1,\dots,h_l\in\R[x]$, consider the following primal-dual semidefinite programs associated with $f$, $g=(g_1,\dots,g_m)$ and $h=(h_1,\dots,h_l)$:
\begin{equation}\label{eq:mom.relax}
\begin{array}{rl}
\tau_k(f,g,h):=\inf\limits_y& L_y(f)\\
\text{s.t} &M_k(y)\succeq 0\,,\,M_{k-d_j}(g_jy)\succeq 0\,,\,j=1,\dots,m\,,\\
&M_{k-r_t}(h_ty)=0\,,\,t=1,\dots,l\,,\,y_0=1\,,
\end{array}
\end{equation}
\begin{equation}\label{eq:sos.relax}
\begin{array}{rl}
\rho_k(f,g,h):=\sup\limits_{\xi,G_j,u_t} & \xi\\
\text{s.t} & G_j\succeq 0\,,\\
&f-\xi=v_k^\top G_0v_k+\sum_{j=1}^m g_jv_{k-d_j}^\top G_jv_{k-d_j}\\
&\qquad\qquad+\sum_{t=1}^l h_tu_t^\top v_{2r_t}\,,\\
\end{array}
\end{equation}
where $d_j=\lceil \deg(g_j)/2\rceil$ and $r_t=\lceil \deg(h_t)/2\rceil$.
Using \cite[Lemma 15]{mai2022exact}, we obtain 
\begin{equation}\label{eq:equi.sos}
\rho_k(f,g,h)=\sup_{\xi\in\R}\{ \xi\,:\,f-\xi\in Q_k(g)[x]+I_k(h)[x]\}\,.
\end{equation}
We call primal-dual semidefinite programs \eqref{eq:mom.relax}-\eqref{eq:sos.relax} the Moment/Sum-of-squares relaxations of order $k$ for problem \eqref{eq:pop.equality}.

We state in the following lemma some recent results involving the Moment/Sum-of-squares relaxations:
\begin{lemma}\label{lem:mom.sos}
Let $f\in\R_d[x]$, $g_1,\dots,g_m\in\R[x]$ and $h_1,\dots,h_l\in\R_{d+1}[x]$. 
Let $\bar f^\star$ be as in \eqref{eq:pop.equality} with $g=(g_1,\dots,g_m)$ and $h=(h_1,\dots,h_l)$. 
Then the following statements hold:
\begin{enumerate}
\item For every $k\in\N$, $\tau_k(f,g,h)\le \tau_{k+1}(f,g,h)$ and $\rho_k(f,g,h)\le \rho_{k+1}(f,g,h)$.
\item For every $k\in\N$, $\rho_k(f,g,h)\le \tau_{k}(f,g,h)\le \bar f^\star$.
\item If $S(g)\cap V(h)$ has non-empty interior, for $k\in\N$ sufficient large, the Slater condition holds for the Moment relaxation \eqref{eq:mom.relax} of order $k$.
\item If $S(g)\cap V(h)$ satisfies the Archimedean condition, $\rho_k(f,g,h)\to \bar f^\star$ as $k\to \infty$.
\item If there exists $R>0$ such that $g_m+h_l=R-x_1^2-\dots-x_n^2$, for $k\in\N$ sufficient large, the Slater condition holds for the SOS relaxation \eqref{eq:sos.relax} of order $k$.
\item If there exists $q\in Q_w(g)[x]$ with $2w\ge d+1$ such that $f-\bar f^\star-q$ vanishes on $V(h)$, then  $\rho_r(f,g,h)=\bar f^\star$ with $r=b(n,2w,l+1)/2$.
\item If there exists $q\in \Sigma_w^2[x]$ with $2w\ge d+1$ such that $f-\bar f^\star-q$ vanishes on $V(h)\cap S(g)$, then $\rho_k(f,\Pi g,h)=\bar f^\star$ with $r=b(n,2w,m+l+1)/2$.
\end{enumerate}
\end{lemma}
\begin{proof}
The proofs of the first five statements can be found in \cite{mai2022exact}.
Let us prove the sixth statement. 
Let $u=f-\bar f^\star-q$.
By assumption,we get $u\in\R_{2w}[x]$ and $u=0$ on $V(h)$. 
Set $s=2\lfloor r/(2w)\rfloor$.
From this, Lemma \ref{lem:pos2} says that there exist  $\sigma \in \Sigma_r^2[x]$ such that $u^{2s} + \sigma \in I_r(h)[x]$.
Let $c=\frac{1}{2s}$. 
Then it holds that $1+t+ct^{2s}\in\Sigma^2_s[t]$.
Thus for all $\varepsilon>0$, we have
\begin{equation}
f-\bar f^\star+\varepsilon=q + \varepsilon(1+\frac{u}{\varepsilon}+c\left(\frac{u}{\varepsilon}\right)^{2s})-c\varepsilon^{1-2s}(u^{2s} + \sigma ) +c\varepsilon^{1-2s}\sigma\in Q_r(g)[x]+I_r(h)[x]\,.
\end{equation}
Then we for all $\varepsilon>0$, $\bar f^\star-\varepsilon$ is a feasible solution of \eqref{eq:equi.sos} of the value $\rho_r(f,g,h)$.
It gives $\rho_r(f,g,h)\ge \bar f^\star-\varepsilon$, for all $\varepsilon>0$, and, in consequence, we get $\rho_r(f,g,h)\ge \bar f^\star$.
Using the second statement, we obtain that $\rho_r(f,g,h)= \bar f^\star$, yielding the sixth statement.

We prove the final statement.
Let $u=f-\bar f^\star-q$.
By assumption,we get  $u=0$ on $S(g)\cap V(h)$. 
Set $s=2\lfloor r/(2w)\rfloor$.
From this, Lemma \ref{lem:pos2} says that there exist  $\eta \in P_r(g)[x]$ such that $u^{2s} + \eta \in I_r(h)[x]$.
Let $c=\frac{1}{2s}$. 
Then it holds that $1+t+ct^{2s}\in\Sigma_s^2[t]$.
Thus for all $\varepsilon>0$, we have
\begin{equation}
f-\bar f^\star+\varepsilon=q + \varepsilon(1+\frac{u}{\varepsilon}+c\left(\frac{u}{\varepsilon}\right)^{2s})-c\varepsilon^{1-2s}(u^{2s} + \eta) +c\varepsilon^{1-2s}\eta\in P_r(g)[x]+I_r(h)[x]\,.
\end{equation}
Analysis similar to that in the proof of the third statement shows  $\rho_r(f,\Pi g,h)= \bar f^\star$, yielding the final statement.
\end{proof}

\subsubsection{The case with denominators}
Given $k\in\N$ and $f,g_1,\dots,g_m,h_1,\dots,h_l,\theta\in\R[x]\backslash\{0\}$, consider the following primal-dual semidefinite programs associated with $f$, $g=(g_1,\dots,g_m)$ and $h=(h_1,\dots,h_l)$:
\begin{equation}\label{eq:mom.relax.deno}
\begin{array}{rl}
\tau_k(f,g,h,\theta):=\inf\limits_y& L_y(\theta^{\eta(k,f,\theta)} f)\\
\text{s.t} &M_k(y)\succeq 0\,,\,M_{k-d_j}(g_jy)\succeq 0\,,\,j=1,\dots,m\,,\\
&M_{k-r_t}(h_ty)=0\,,\,t=1,\dots,l\,,\,L_y(\theta^{\eta(k,f,\theta)})=1\,,
\end{array}
\end{equation}
\begin{equation}\label{eq:sos.relax.deno}
\begin{array}{rl}
\rho_k(f,g,h,\theta):=\sup\limits_{\xi,G_j,u_t} & \xi\\
\text{s.t} & G_j\succeq 0\,,\\
&\theta^{\eta(k,f,\theta)}(f-\xi)=v_k^\top G_0v_k+\sum_{j=1}^m g_jv_{k-d_j}^\top G_jv_{k-d_j}\\
&\qquad\qquad+\sum_{t=1}^l h_tu_t^\top v_{2r_t}\,,\\
\end{array}
\end{equation}
where $d_j=\lceil \deg(g_j)/2\rceil$, $r_t=\lceil \deg(h_t)/2\rceil$, and 
\begin{equation}
\eta(k,f,\theta):=2\lfloor\frac{2k-\deg(f)}{2\deg(\theta)}\rfloor\,.
\end{equation}
Using \cite[Lemma 15]{mai2022exact}, we obtain 
\begin{equation}\label{eq:equi.sos.deno}
\rho_k(f,g,h,\theta):=\sup_{\xi\in\R}\{ \xi\,:\,\theta^{\eta(k,f,\theta)}(f-\xi)\in Q_k(g)[x]+I_k(h)[x]\}\,.
\end{equation}
Developed in \cite{mai2021positivity}, primal-dual semidefinite programs \eqref{eq:mom.relax.deno}-\eqref{eq:sos.relax.deno} are another type of the Moment/Sum-of-squares relaxations of order $k$ for problem \eqref{eq:pop.equality}.

We state in the following lemma some recent results involving this type of Moment/Sum-of-squares relaxations:
\begin{lemma}\label{lem:mom.sos.deno}
Let $f\in\R_d[x]$, $g_1,\dots,g_m\in\R[x]$, $h_1,\dots,h_l\in\R_{d+1}[x]$, and $\theta\in\R_{2u}[x]\backslash\{0\}$. 
Let $\bar f^\star$ be as in \eqref{eq:pop.equality} with $g=(g_1,\dots,g_m)$ and $h=(h_1,\dots,h_l)$. 
Then the following statements hold:
\begin{enumerate}
\item For every $k\in\N$, $\tau_k(f,g,h,\theta)\le \tau_{k+1}(f,g,h,\theta)$, $\rho_k(f,g,h,\theta)\le \rho_{k+1}(f,g,h,\theta)$, and $\rho_k(f,g,h,\theta)\le \tau_{k}(f,g,h,\theta)$.
\item If $S(g)\cap V(h)$ has non-empty interior, for $k\in\N$ sufficient large, the Slater condition holds for the Moment relaxation \eqref{eq:mom.relax} of order $k$.
\item If one of the following two conditions holds:
\begin{enumerate}
\item problem \eqref{eq:pop.equality} has an optimal solution $x^\star$ such that $\theta(x^\star)>0$;
\item there exists a sequence of feasible solutions  $(x^{(t)})_{t\in\N}$ for problem \eqref{eq:pop.equality} such that $(f(x^{(t)}))_{t\in\N}$ converges to $\bar f^\star$ and  $\theta(x^{(t)})>0$,
\end{enumerate}
then for every $k\in\N$, $\rho_k(f,g,h,\theta)\le \bar f^\star$.
\item If $\rho_k(f,g,h,\theta)\le \bar f^\star$, for every $k\in\N$, and there exists $q\in Q_{w}(g)[x]$ with $2(w+u)\ge d+1$ such that $\theta(f-\bar f^\star-q)$ vanishes on $V(h)$, then  $\rho_r(f,g,h,\theta)=\bar f^\star$ with $r=b(n,2(w+u),l+1)/2$.
\end{enumerate}
\end{lemma}
\begin{proof}
The proofs of the first two statements are similar to the ones of Lemma \ref{lem:mom.sos} (see also in \cite{mai2021positivity}).
Let us prove the third statement.
Let $\varepsilon>0$. 
By \eqref{eq:equi.sos.deno}, we get 
\begin{equation}\label{eq:sup.prop}
 \theta^{\eta(k,f,\theta)}(f-\rho_k(f,g,h,\theta)+\varepsilon)\in Q_k(g)[x]+I_k(h)[x]\,.
\end{equation}
Consider the following two cases:
\begin{itemize}
\item Assume that the condition (a) holds. 
By \eqref{eq:sup.prop}, it implies that $\theta(x^\star)^{\eta(k,f,\theta)}(\bar f^\star-\rho_k(f,g,h,\theta)+\varepsilon)\ge 0$.
Since $\theta(x^\star)>0$, $\bar f^\star\ge \rho_k(f,g,h,\theta)-\varepsilon$.
The result follows since $\varepsilon$ is arbitrary.
\item Assume that the condition (b) holds. By \eqref{eq:sup.prop}, it implies that $\theta(x^{(t)})^{\eta(k,f,\theta)}(f(x^{(t)})-\rho_k(f,g,h,\theta)+\varepsilon)\ge 0$.
Since $\theta(x^{(t)})>0$, $f(x^{(t)})\ge \rho_k(f,g,h,\theta)-\varepsilon$.
As $t$ goes to infinity, $\bar f^\star\ge \rho_k(f,g,h,\theta)-\varepsilon$.
The result follows since $\varepsilon$ is arbitrary.
\end{itemize}

Let us prove the fourth statement. 
Let $p=f-\bar f^\star-q$.
By assumption,we get $\theta p\in\R_{2(w+u)}[x]$ and $\theta p=0$ on $V(h)$. 
Set $s=2\lfloor r/(2w)\rfloor$.
From this, Lemma \ref{lem:pos2} says that there exist  $\sigma \in \Sigma_r^2[x]$ such that $(\theta p)^{2s} + \sigma \in I_r(h)[x]$.
Let $c=\frac{1}{2s}$. 
Then it holds that $1+t+ct^{2s}\in\Sigma^2_s[t]$.
Thus for all $\varepsilon>0$, we have
\begin{equation}
\theta(f-\bar f^\star+\varepsilon)=\theta q + \varepsilon \theta (1+\frac{p}{\varepsilon}+c\left(\frac{p}{\varepsilon}\right)^{2s})-c(\varepsilon\theta)^{1-2s}((\theta p)^{2s} + \sigma ) +c(\varepsilon\theta)^{1-2s}\sigma\,.
\end{equation}
It implies that $\theta^{2s}(f-\bar f^\star+\varepsilon)\in Q_r(g)[x]+I_r(h)[x]$.
Then we for all $\varepsilon>0$, $\bar f^\star-\varepsilon$ is a feasible solution of \eqref{eq:equi.sos} of the value $\rho_r(f,g,h,\theta)$.
It gives $\rho_r(f,g,h,\theta)\ge \bar f^\star-\varepsilon$, for all $\varepsilon>0$, and, in consequence, we get $\rho_r(f,g,h,\theta)\ge \bar f^\star$.
By assumption, we obtain that $\rho_r(f,g,h,\theta)= \bar f^\star$, yielding the final statement.
\end{proof}



\section{Explicit degree bounds for the representation based on the Fritz John conditions}
\label{sec:degree.bound}
This section provides the degree bounds for our five Nichtnegativstellens\"atze based on the Fritz John conditions.

Let $\bar\lambda=(\lambda_0,\lambda_1,\dots,\lambda_m)$ be a vector of $m$ variables.
Set $\lambda:=(\lambda_1,\dots,\lambda_m)$.
\subsection{The case of arbitrary multipliers}
 We state the first main result in the following theorem:
\begin{theorem}\label{theo:rep}
Let $f,g_1,\dots,g_m\in\R_d[x]$ with $d\ge 1$. 
Assume that $f$ is non-negative on $S(g)$  with $g:=(g_1,\dots,g_m)$ and $f(C(g))$ is finite. 
Set 
\begin{equation}\label{eq:def.w}
w:=\frac{1}{2}\times b(n+m+1,d+1,2m+n+3)+d\times {c(n+m+1,d+1,n+m+1)}\,.
\end{equation}
Then there exists $q\in P_w(g)[x,\bar \lambda]$ such that $f-q$ vanishes on $V(h_{\FJ})$, where $\bar\lambda:=(\lambda_0,\dots,\lambda_m)$ and
\begin{equation}\label{eq:.polyFJ}
    h_{\FJ}:=(\lambda_0\nabla f-\sum_{j=1}^m \lambda_j \nabla g_j,\lambda_1g_1,\dots,\lambda_mg_m,1-\sum_{j=0}^m\lambda_j^2)\,.
\end{equation}
\end{theorem}
\begin{proof}
Using Lemma \ref{lem:num.connected}, we decompose $V(h_{\FJ})$ into semi-algebraically path connected components:
$Z_1,\dots,Z_s$ with  
\begin{equation}\label{eq:bound.on.s}
s\le c(n+m+1,d+1,n+m+1)\,,
\end{equation}
since each entry of $h_{\FJ}$ has degree at most $d+1\ge 2$.
Accordingly \cite[Lemma 13]{mai2022exact} shows that $f$ is constant on each $Z_i$.
Thus $f(V(h_{\FJ}))$ is finite.
Set $r=|f(V(h_{\FJ}))|$.
From \eqref{eq:bound.on.s}, we get 
\begin{equation}\label{eq:ineq}
r\le s\le c(n+m+1,d+1,n+m+1)\,.
\end{equation}
Set 
\begin{equation}
u:=b(n+m+1,d+1,2m+n+3)/2\,.
\end{equation}
By using Lemma \ref{lem:quadra}, there exist $q\in P_\xi(g)[x,\bar\lambda]$ with $\xi=dr+u$ such that $f - q$ vanishes on $V(h_{\FJ})$.
By \eqref{eq:ineq} and \eqref{eq:def.w}, we obtain $\xi\le w$, and hence $q\in P_w(g)[x,\bar\lambda]$.
\end{proof} 
\begin{remark}
Here $h_{\FJ}$ includes polynomials from the Fritz John conditions \eqref{eq:FJcond}.
In Theorem \ref{theo:rep}, if we assume further that the ideal generated by $h_{\FJ}$ is real radical (see in \cite{mai2022exact}), then $f-q\in I_r(h_{\FJ})$ for some $r\in\N$. 
However, this further assumption does not play any role in this paper when we apply the representation in Theorem \ref{theo:rep} for polynomial optimization.
\end{remark}

Our second main result is as follows:
\begin{theorem}\label{theo:rep2}
Let $f,g_1,\dots,g_m\in\R_d[x]$ with $d\ge 1$.  
Assume that $f$ is non-negative on $S(g)$  with $g:=(g_1,\dots,g_m)$ and $f(C(g)\cap S(g))$ is finite.
Set 
\begin{equation}\label{eq:def.w2}
w:=d\times(c(n+m+1,d+1,n+2m+1)-1)\,.
\end{equation} 
Then there exists $q\in \Sigma_w^2[x,\bar \lambda]$ such that $f-q$ vanishes on $(S(g)\times\R^{m+1})\cap V(h_{\FJ})$, where $\bar\lambda:=(\lambda_0,\dots,\lambda_m)$ and $h_{\FJ}$ is defined as in \eqref{eq:.polyFJ}.
\end{theorem}
\begin{proof}
Using Lemma \ref{lem:num.connected}, we decompose $(S(g)\times\R^{m+1})\cap V(h_{\FJ})$ into semi-algebraically path connected components:
$Z_1,\dots,Z_s$ with  
\begin{equation}
s\le c(n+m+1,d+1,n+2m+1)\,,
\end{equation}
since each entry of $h_{\FJ}$ (resp. $g$) has degree at most $d+1\ge 2$ (resp. $d$).
Accordingly \cite[Lemma 19]{mai2022exact} shows that $f$ is constant on $Z_i$.
Thus the set $f((S(g)\times\R^{m+1})\cap V(h_{\FJ}))$ has $r$ elements and we get
\begin{equation}\label{eq:ineq2}
r\le s\le c(n+m+1,d+1,n+2m+1)\,.
\end{equation}
By using Lemma \ref{lem:quadra2}, there exist $q\in \Sigma^2_\xi[x,\bar\lambda]$ with $\xi=d(r-1)$ such that $f - q$ vanishes on $(S(g)\times\R^{m+1})\cap V(h_{\FJ})$.
By \eqref{eq:ineq2} and \eqref{eq:def.w2}, we obtain $\xi\le w$, and hence $q\in \Sigma^2_w[x,\bar\lambda]$.
\end{proof}

\subsection{The case of nonnegative multipliers}
We state the third main result in the following theorem:
\begin{theorem}\label{theo:rep.plus}
Let $f,g_1,\dots,g_m\in\R_{d}[x]$ with $d\in\N$. 
Assume that $f$ is non-negative on $S(g)$  with $g:=(g_1,\dots,g_m)$ and $f(C^+(g))$ is finite. 
Set 
\begin{equation}\label{eq:def.w3}
w:=\frac{1}{2}\times b(n+m+1,d+2,2m+n+3)+d\times{c(n+m+1,d+2,n+m+1)}\,.
\end{equation}
Then there exists $q\in P_w(g)[x,\bar \lambda]$ such that $f-q$ vanishes on $V(h_{\FJ}^+)$, where $\bar\lambda:=(\lambda_0,\dots,\lambda_m)$ and
\begin{equation}\label{eq:.polyFJ.plus}
    h_{\FJ}^+:=(\lambda_0^2\nabla f-\sum_{j=1}^m \lambda_j^2 \nabla g_j,\lambda_1^2g_1,\dots,\lambda_m^2g_m,1-\sum_{j=0}^m\lambda_j^2)\,.
\end{equation}
\end{theorem}
Here $h_{\FJ}^+$ includes polynomials in the right hand side of the Fritz John conditions \eqref{eq:FJcond}.
The fourth main result is stated as follows:
\begin{theorem}\label{theo:rep.plus2}
Let $f,g_1,\dots,g_m\in\R_d[x]$ with $d\in\N$. 
Assume that $f$ is non-negative on $S(g)$  with $g:=(g_1,\dots,g_m)$ and $f(C^+(g)\cap S(g))$ is finite. 
Set 
\begin{equation}\label{eq:def.w4}
w:=d\times(c(n+m+1,d+2,n+2m+1)-1)\,.
\end{equation} 
Then there exists $q\in \Sigma_w[x,\bar \lambda]$ such that $f-q$ vanishes on $S(g)\cap V(h_{\FJ}^+)$, where $\bar\lambda:=(\lambda_0,\dots,\lambda_m)$ and $h_{\FJ}^+$ is defined as in \eqref{eq:.polyFJ.plus}.
\end{theorem}

To prove Theorem \ref{theo:rep.plus} (resp. Theorem \ref{theo:rep.plus2}), we do similarly to the proof of Theorem \ref{theo:rep} (resp. Theorem \ref{theo:rep2}) by replacing $C(g)$ and $h_{\FJ}$ with $C^+(g)$ and $h_{\FJ}^+$, respectively.
Note that each entry of $h_{\FJ}^+$ has degree at most $d+2$.
\begin{remark}
The two sets of critical points $C(g)$ and $C^+(g)$ include the set of points $W(f,g)$ at which the Fritz John conditions hold but the Karush--Kuhn--Tucker conditions do not hold for problem \eqref{eq:pop}. 
The assumptions that the images of $C(g)$ and $C^+(g)$ under $f$ are finite allow us to obtain the representations of $f$ in the case of $W(f,g)\ne \emptyset$.
\end{remark}
\subsection{The case with denominators}
We state the fifth main result in the following theorem, which does not require any assumption on the set of critical points:
\begin{theorem}\label{theo:rep.deno}
Let $f,g_1,\dots,g_m\in\R_d[x]$ with $d\ge 1$.
Assume that $f$ is non-negative on $S(g)$  with $g:=(g_1,\dots,g_m)$. 
Set 
\begin{equation}\label{eq:def.w5}
w:=\frac{1}{2}\times b(n+m+1,d+1,2m+n+3)+2d\times {c(n+m+1,d+1,n+m+1)}\,.
\end{equation} 
Then there exists $q\in P_w(g)[x,\bar \lambda]$ such that $\lambda_0(f-q)$ vanishes on $V(h_{\FJ})$, where $\bar\lambda:=(\lambda_0,\dots,\lambda_m)$ and $h_{\FJ}$ is defined as in \eqref{eq:.polyFJ}.
\end{theorem}
\begin{proof}
Using Lemma \ref{lem:num.connected}, we decompose $V(h_{\FJ})\backslash \{\lambda_0=0\}$ into semi-algebraically path connected components:
$Z_1,\dots,Z_s$ with  
\begin{equation}\label{eq:bound.on.s.deno}
s\le 2\times c(n+m+1,d+1,n+m+1)\,.
\end{equation}
It is because each entry of $h_{\FJ}$ has degree at most $d+1\ge 2$ and 
\begin{equation}
V(h_{\FJ})\backslash \{\lambda_0=0\}=(V(h_{\FJ})\cap\{\lambda_0>0\})\cup (V(h_{\FJ})\cap\{\lambda_0<0\})\,.
\end{equation}
Accordingly \cite[Lemma 24]{mai2022exact} shows that $f$ is constant on each $Z_i$.
Thus $f(V(h_{\FJ})\backslash \{\lambda_0=0\})$ is finite.
Set $r=|f(V(h_{\FJ})\backslash \{\lambda_0=0\})|$.
From \eqref{eq:bound.on.s}, we get 
\begin{equation}\label{eq:ineq.deno}
r\le s\le 2\times c(n+m+1,d+1,n+m+1)\,.
\end{equation}
Set 
\begin{equation}
u:=b(n+m+1,d+1,2m+n+3)/2\,.
\end{equation}
By using Lemma \ref{lem:quadra.deno}, there exist $q\in P_\xi(g)[x,\bar\lambda]$ with $\xi=dr+u$ such that $f - q$ vanishes on $V(h_{\FJ})\backslash \{\lambda_0=0\}$.
It implies that $\lambda_0(f - q)$ vanishes on $V(h_{\FJ})$.
By \eqref{eq:ineq.deno} and \eqref{eq:def.w5}, we obtain $\xi\le w$, and hence $q\in P_w(g)[x,\bar\lambda]$.
\end{proof}

\begin{remark}
We have provided the degree bounds for the representations involving preorderings in Theorems \ref{theo:rep}, \ref{theo:rep2}, \ref{theo:rep.plus}, \ref{theo:rep.plus2}, and \ref{theo:rep.deno}.
Our bounds are extremely large because of the exponential bounds for Krivine--Stengle's Nichtnegativstellensatz and the number of connected components of a basic semi-algebraic set.
Fortunately, our bounds only depend on the number of variables, the number of polynomials defining the basic semi-algebraic set, and the degrees of the input polynomials.

In contrast, our degree bound for the representations involving quadratic modules in Theorems \ref{theo:rep.quadra.module} and \ref{theo:rep.quadra.module.deno} stated below is implicit and depends on all information (including the coefficients) of the input polynomials.
To achieve it, we apply the degree bound of Putinar's Positivestellensatz analyzed by Baldi--Mourrain \cite{baldi2021moment} under the finiteness assumption for the image of the set of critical points $C(g)$ under $f$.
\end{remark}

\section{Convergence rate for exact polynomial optimization based on the Fritz John conditions}
\label{sec:application}

This section presents the main application of the degree bounds in Theorems \ref{theo:rep}, \ref{theo:rep2}, \ref{theo:rep.plus}, \ref{theo:rep.plus2}, and \ref{theo:rep.deno} to analyze the convergence rate of the corresponding hierarchies of semidefinite relaxations for a polynomial optimization problem.


\subsection{The case of arbitrary multipliers}

\begin{theorem}\label{theo:pop}
Let $f,g_1,\dots,g_m\in\R_d[x]$. 
Let $f^\star$ be as in problem \eqref{eq:pop} with $g=(g_1,\dots,g_m)$.
Assume that problem \eqref{eq:pop} has a global minimizer and $f(C(g))$ is finite.
Let $h_{\FJ}$ be as in \eqref{eq:.polyFJ} and let $w$ be as in \eqref{eq:def.w5}.
Set 
\begin{equation}\label{eq:def.r}
r=b(n+m+1,2w,m+n+2)/2\,.
\end{equation}
Then $\rho_r(f,\Pi g,h_{\FJ})=f^\star$, where $\Pi g$ is defined as in \eqref{eq:prod.g}.
\end{theorem}
\begin{proof}
Since $S(g)=S(\Pi g)$, \cite[Lemma 17]{mai2022exact} implies that
\begin{equation}\label{eq:equivalent.prob}
\begin{array}{rl}
f^\star:=\min\limits_{x,\bar\lambda}& f(x)\\
\text{s.t.}& x\in S(\Pi g)\,,\,(x,\bar\lambda)\in V(h_{\FJ})\,,
\end{array}
\end{equation}
By assumption, Theorem \ref{theo:rep} yields that there exists $q\in P_w(g)[x,\bar \lambda]=Q_w(\Pi g)[x,\bar \lambda]$ such that $f-f^\star-q$ vanishes on $V(h_{\FJ})$.
Applying the sixth statement of Lemma \ref{lem:mom.sos} (by replacing $g$ with $\Pi g$), we obtain the conclusion.
\end{proof}

\begin{theorem}\label{theo:pop2}
Let $f,g_1,\dots,g_m\in\R_d[x]$. 
Let $f^\star$ be as in problem \eqref{eq:pop} with $g=(g_1,\dots,g_m)$.
Assume that problem \eqref{eq:pop} has a global minimizer and $f(C(g)\cap S(g))$ is finite.
Let $h_{\FJ}$ be as in \eqref{eq:.polyFJ} and let $w$ be as in \eqref{eq:def.w2}.
Set $r$ as in \eqref{eq:def.r}.
Then $\rho_r(f,\Pi g,h_{\FJ})=f^\star$, where $\Pi g$ is defined as in \eqref{eq:prod.g}.
\end{theorem}
\begin{proof}
Note that \cite[Lemma 17]{mai2022exact} implies that
\begin{equation}\label{eq:equi.prob1}
\begin{array}{rl}
f^\star:=\min\limits_{x,\bar\lambda}& f(x)\\
\text{s.t.}& x\in S( g)\,,\,(x,\bar\lambda)\in V(h_{\FJ})\,,
\end{array}
\end{equation}
By assumption, Theorem \ref{theo:rep2} yields that there exists $q\in \Sigma^2_w[x,\bar \lambda]$ such that $f-f^\star-q$ vanishes on $(S(g)\cap\R^{m+1})\times V(h_{\FJ})$.
Applying the final statement of Lemma \ref{lem:mom.sos}, we obtain the conclusion.
\end{proof}
\subsection{The case of nonnegative multipliers}
\begin{theorem}\label{theo:pop.plus}
Let $f,g_1,\dots,g_m\in\R_d[x]$. 
Let $f^\star$ be as in problem \eqref{eq:pop} with $g=(g_1,\dots,g_m)$.
Assume that problem \eqref{eq:pop} has a global minimizer and $f(C^+(g))$ is finite.
Let $h_{\FJ}^+$ be as in \eqref{eq:.polyFJ} and let $w$ be as in \eqref{eq:def.w3}.
Set $r$ as in \eqref{eq:def.r}.
Then $\rho_r(f,\Pi g,h_{\FJ})=f^\star$, where $\Pi g$ is defined as in \eqref{eq:prod.g}.
\end{theorem}
The proof of Theorem \ref{theo:pop.plus}, which is based on Theorem \ref{theo:rep.plus}, is similar to the one of Theorem \ref{theo:pop}.
\begin{theorem}\label{theo:pop.plus2}
Let $f,g_1,\dots,g_m\in\R_d[x]$. 
Let $f^\star$ be as in problem \eqref{eq:pop} with $g=(g_1,\dots,g_m)$.
Assume that problem \eqref{eq:pop} has a global minimizer and $f(C^+(g)\cap S(g))$ is finite.
Let $h_{\FJ}^+$ be as in \eqref{eq:.polyFJ} and let $w$ be as in \eqref{eq:def.w4}.
Set $r$ as in \eqref{eq:def.r}.
Then $\rho_r(f,\Pi g,h_{\FJ})=f^\star$, where $\Pi g$ is defined as in \eqref{eq:prod.g}.
\end{theorem}
The proof of Theorem \ref{theo:pop.plus2} based on Theorem \ref{theo:rep.plus2} is similar to the one of Theorem \ref{theo:pop2}.

\begin{remark}
It is worth pointing out in Theorems \ref{theo:pop}, \ref{theo:pop2}, \ref{theo:pop.plus}, and \ref{theo:pop.plus2} that the order $r\in\N$ satisfying $\rho_r(f,\Pi g,h_{\FJ})=f^\star$ depends on $n$ (the number of variables), $m$ (the number of inequality constraints $g_j$), and $d$ (the upper bound on the degrees of $f,g_j$) but does not depend on the coefficients of $f,g_j$.
\end{remark}

\begin{remark}
Under the Archimedean condition (resp. the compactness assumption) for $S(g)$,  the sequence $(\rho_k(f,g,h_{\FJ}))_{k\in\N}$ (resp. $(\rho_k(f,\Pi g,h_{\FJ}))_{k\in\N}$) converges to $f^\star$ faster than the sequence $(\rho_k(f,g,0))_{k\in\N}$ (resp. $(\rho_k(f,\Pi g,0))_{k\in\N}$).
It is because of the following inequalities:
\begin{equation}
\rho_k(f,g,0)\le\rho_k(f,g,h_{\FJ})\le f^\star\quad\quad(\text{resp. }\rho_k(f,\Pi g,0)\le\rho_k(f,\Pi g,h_{\FJ})\le f^\star)\,.
\end{equation}
Thus we can obtain the convergence rate for $(\rho_k(f,g,h_{\FJ}))_{k\in\N}$ (resp. $(\rho_k(f,\Pi g,h_{\FJ}))_{k\in\N}$) from the available convergence rate of $(\rho_k(f,g,0))_{k\in\N}$ (resp. $(\rho_k(f,\Pi g,0))_{k\in\N}$) in \cite{fang2020sum,laurent2021effective,slot2021sum,baldi2021moment}.
However, these rates do not give the finite convergence for the sequence $(\rho_k(f,g,h_{\FJ}))_{k\in\N}$ (resp. $(\rho_k(f,\Pi g,h_{\FJ}))_{k\in\N}$).
Moreover, the computational complexity for the semidefinite relaxation of each value $\rho_k(f,g,h_{\FJ})$ (resp. $\rho_k(f,\Pi g,h_{\FJ})$) grows more quickly than the one of the value $\rho_k(f,g,0)$ (resp. $\rho_k(f,\Pi g,0)$) as $k$ increases.
It is because to obtain the relaxation of the value $\rho_k(f,g,h_{\FJ})$ (resp. $\rho_k(f,\Pi g,h_{\FJ})$), we utilize $m+1$ additional variables and $n+m+1$ additional equality constraints to the original polynomial optimization problem \eqref{eq:pop}.
\end{remark}

\subsection{The case with denominators}

\begin{theorem}\label{theo:pop.deno}
Let $f,g_1,\dots,g_m\in\R_d[x]$. 
Let $f^\star$ be as in problem \eqref{eq:pop} with $g=(g_1,\dots,g_m)$.
Let $h_{\FJ}$ be as in \eqref{eq:.polyFJ} and let $w$ be as in \eqref{eq:def.w5}.
Assume that problem \eqref{eq:pop} has global minimizer  and one of the following two conditions holds:
\begin{enumerate}
\item  the Karush--Kuhn--Tucker conditions hold for problem \eqref{eq:pop} at $x^\star$;
\item there exists a sequence of points $(x^{(t)},\bar \lambda^{(t)})_{t\in\N}$ in $(S(g)\times \R^{m+1})\cap V(h_{\FJ})$ such that $(f(x^{(t)}))_{t\in\N}$ converges to $\bar f^\star$ and  $\lambda_0^{(t)}>0$,
\end{enumerate}
Set 
\begin{equation}\label{eq:def.r.deno}
r=b(n+m+1,2(w+1),m+n+2)/2\,.
\end{equation}
Then $\rho_r(f,\Pi g,h_{\FJ},\lambda_0)=f^\star$, where $\Pi g$ is defined as in \eqref{eq:prod.g}.
\end{theorem}
\begin{proof}
Since $S(g)=S(\Pi g)$, \cite[Lemma 17]{mai2022exact} implies \eqref{eq:equivalent.prob}.
Note that the first condition implies that the Fritz John conditions hold for problem \eqref{eq:pop} at $x^\star$ with multipliers $\bar\lambda^\star$ such that $\lambda_0^\star>0$, so that $(x^\star,\bar\lambda^\star)$ is a global minimizer for problem \eqref{eq:equivalent.prob} with $\lambda_0^\star>0$.
The second condition implies that each $(x^{(t)},\bar \lambda^{(t)})$ is a feasible solution for problem \eqref{eq:equivalent.prob} such that $(f(x^{(t)}))_{t\in\N}$ converges to $\bar f^\star$ and  $\lambda_0^{(t)}>0$.
The third statement of Lemma \ref{lem:mom.sos.deno} says that $\rho_k(f,\Pi g,h_{\FJ},\lambda_0)\le f^\star$, for every $k\in\N$.
In addition, Theorem \ref{theo:rep.deno} yields that there exists $q\in P_w(g)[x,\bar \lambda]=Q_w(\Pi g)[x,\bar \lambda]$ such that $\lambda_0(f-f^\star-q)$ vanishes on $V(h_{\FJ})$.
Applying the final statement of Lemma \ref{lem:mom.sos.deno} (by replacing $g,h,\theta$ with $\Pi g,h_{\FJ},\lambda_0$), we obtain the conclusion.
\end{proof}
\begin{remark}
As shown by Freund in \cite{freund2016optimality}, the Karush--Kuhn--Tucker conditions hold for problem \eqref{eq:pop} at its global minimizers if one of the following conditions holds:
\begin{enumerate}
\item $g$ can be decomposed as $g=(g^{(1)},g^{(2)},-g^{(2)})$, $S(g^{(1)})$ has non-empty interior, and $g_j$, $j=1,\dots,m$, are concave;
\item $g_j$, $j=1,\dots,m$, are linear;
\item $f$ is a pseudoconvex and $g_j$, $j=1,\dots,m$, are quasiconcave (see the definitions of pseudoconvexity and quasiconcavity in \cite{freund2016optimality}).
\end{enumerate}
Thanks to Theorem \ref{theo:pop.deno} (as well as Theorems \ref{theo:pop.KKT}, \ref{theo:pop.KKT.plus}, and \ref{theo:pop.quadra.deno} stated below), we can compute the exact optimal value $f^\star$ for problem \eqref{eq:pop} by using semidefinite programming under one of the above three conditions.
\end{remark}

\section{Explicit degree bounds for the representation based on the Karush--Kuhn--Tucker conditions}
\label{sec:bound.KKT}
\subsection{The case of arbitrary multipliers}
We prove the following lemma similar to \cite[Lemma 3.3]{demmel2007representations} by using the tools from real algebraic geometry (instead of the ones from complex algebraic geometry):
\begin{lemma}\label{lem:constant.KKT}
Let $f,g_1,\dots,g_m\in\R[x]$. 
Set
\begin{equation}\label{eq:.polyKKT}
    h_{\KKT}:=(\nabla f-\sum_{j=1}^m \lambda_j \nabla g_j,\lambda_1g_1,\dots,\lambda_mg_m)\,.
\end{equation}
Let $W$ be a semi-algebraically path connected component of $V(h_{\KKT})$. Then $f$ is constant on $W$.
\end{lemma}
\begin{proof} 
Recall $\lambda:=(\lambda_1,\dots,\lambda_m)$.
Choose two arbitrary points $(x^{(0)},\lambda^{(0)})$, $(x^{(1)},\lambda^{(1)})$ in $W$. 
We claim that $f(x^{(0)}) = f(x^{(1)})$.

It is sufficient to assume that both $(x^{(0)},\lambda^{(0)})$ and $(x^{(1)},\lambda^{(1)})$ are non-singular points.
If at least one of $(x^{(0)},\lambda^{(0)})$ and $(x^{(1)},\lambda^{(1)})$ is singular, we choose arbitrarily close non-singular points to approximate $(x^{(0)},\lambda^{(0)})$ and $(x^{(1)},\lambda ^{(1)})$ then we apply the continuity of $f$ to obtain $f(x^{(0)}) = f(x^{(1)})$.
It is because the set of non-singular points of $W$ is dense and open in $W$.

If a manifold is path-connected, the set of its non-singular points is a manifold that is also path-connected.
By assumption, there exists a continuous piecewise-differentiable path $\phi(\tau) = (x(\tau),\lambda(\tau))$, for $\tau\in[0,1]$, lying inside $W$ such that $\phi(0) = (x^{(0)},\lambda^{(0)})$ and $\phi(1) = (x^{(1)},\lambda^{(1)})$. 
We claim that $\tau\mapsto f(x(\tau))$ is constant on $[0,1]$.
The Lagrangian function
\begin{equation}\label{eq:Lagran.KKT}
    L(x,\lambda) = f(x)+\sum_{j=1}^m \lambda_j g_j (x)\,.
\end{equation}
is equal to $f(x)$ on $V(h_{\KKT})$, which contains $\phi([0,1])$. 
Let $\mu_j(\tau)$ be the principal square root of
$\lambda_j(\tau)$ for $\tau\in[0,1]$, $j=1,\dots,m$. 
By the mean value theorem, it follows that $f (x (0) ) = f (x (1) )$.
We now obtain $f (x^{(0)})$ = $f (x^{(1)})$ and hence $f$ is constant on $W$.
\end{proof}
We state in the following theorem the explicit degree bound for the representation based on  the Karush--Kuhn--Tucker conditions:
\begin{theorem}\label{theo:rep.KKT}
Let $f,g_1,\dots,g_m\in\R_d[x]$ with $d\ge 1$. 
Assume that $f$ is non-negative on $S(g)$  with $g:=(g_1,\dots,g_m)$. 
Set 
\begin{equation}\label{eq:def.w.KKT}
w:=\frac{1}{2}\times b(n+m,d+1,2m+n+1)+d\times {c(n+m,d+1,n+m)}\,.
\end{equation}
Then there exists $q\in P_w(g)[x, \lambda]$ such that $f-q$ vanishes on $V(h_{\KKT})$, where $\lambda:=(\lambda_1,\dots,\lambda_m)$ and $h_{\KKT}$ is defined as in \eqref{eq:.polyKKT}.
\end{theorem} 
\begin{proof}
Using Lemma \ref{lem:num.connected}, we decompose $V(h_{\KKT})$ into semi-algebraically path connected components:
$Z_1,\dots,Z_s$ with  
\begin{equation}\label{eq:bound.on.s.KKT}
s\le c(n+m,d+1,n+m)\,,
\end{equation}
since each entry of $h_{\KKT}$ has degree at most $d+1\ge 2$.
Accordingly Lemma \ref{lem:constant.KKT} shows that $f$ is constant on each $Z_i$.
Thus $f(V(h_{\KKT}))$ is finite.
Set $r=|f(V(h_{\KKT}))|$.
From \eqref{eq:bound.on.s.KKT}, we get 
\begin{equation}\label{eq:ineq.KKT}
r\le s\le c(n+m,d+1,n+m)\,.
\end{equation}
Set 
\begin{equation}
u:=b(n+m,d+1,2m+n+1)/2\,.
\end{equation}
By using Lemma \ref{lem:quadra}, there exists $q\in P_\xi(g)[x,\lambda]$ with $\xi=dr+u$ such that $f - q$ vanishes on $V(h_{\KKT})$.
By \eqref{eq:ineq.KKT} and \eqref{eq:def.w.KKT}, $\xi\le w$, and hence $q\in P_w(g)[x,\lambda]$. 
\end{proof}
We apply Theorem \ref{theo:rep.KKT} for polynomial optimization as follows:
\begin{theorem}\label{theo:pop.KKT}
Let $f,g_1,\dots,g_m\in\R_d[x]$. 
Let $f^\star$ be as in problem \eqref{eq:pop} with $g=(g_1,\dots,g_m)$.
Assume that problem \eqref{eq:pop} has a global minimizer at which the Karush--Kuhn--Tucker conditions hold for this problem.
Let $h_{\KKT}$ be as in \eqref{eq:.polyKKT} and let $w$ be as in \eqref{eq:def.w.KKT}.
Set 
\begin{equation}
r=b(n+m,2w,m+n+1)/2\,.
\end{equation}
Then $\rho_r(f,\Pi g,h_{\KKT})=f^\star$, where $\Pi g$ is defined as in \eqref{eq:prod.g}.
\end{theorem}
\begin{proof}
By assumption, there exists $(x^\star,\lambda^\star)\in V(h_{\KKT})$ such that $x^\star$ is a global minimizer of \eqref{eq:pop}.
Since $S(g)=S(\Pi g)$, it implies that
\begin{equation}
\begin{array}{rl}
f^\star:=\min\limits_{x,\lambda}& f(x)\\
\text{s.t.}& x\in S(\Pi g)\,,\,(x,\lambda)\in V(h_{\KKT})\,,
\end{array}
\end{equation}
By assumption, Theorem \ref{theo:rep.KKT} yields that there exists $q\in P_w(g)[x,\lambda]=Q_w(\Pi g)[x, \lambda]$ such that $f-f^\star-q$ vanishes on $V(h_{\KKT})$.
Applying the sixth statement of Lemma \ref{lem:mom.sos} (by replacing $g$ with $\Pi g$), we obtain the conclusion.
\end{proof}
\subsection{The case of nonnegative multipliers}

A similar proof to one of Lemma \ref{lem:constant.KKT} applies to the following lemma:
\begin{lemma}\label{lem:constant.KKT.plus}
Let $f,g_1,\dots,g_m\in\R[x]$. 
Set
\begin{equation}\label{eq:.polyKKT.plus}
    h_{\KKT}^+:=(\nabla f-\sum_{j=1}^m \lambda_j^2 \nabla g_j,\lambda_1^2g_1,\dots,\lambda_m^2g_m)\,.
\end{equation}
Let $W$ be a semi-algebraically path connected component of $V(h_{\KKT}^+)$. Then $f$ is constant on $W$.
\end{lemma}
We state in the following theorem the explicit degree bound for the representation based on  the Karush--Kuhn--Tucker conditions:
\begin{theorem}\label{theo:rep.KKT.plus}
Let $f,g_1,\dots,g_m\in\R_d[x]$ with $d\in\N$. 
Assume that $f$ is non-negative on $S(g)$  with $g:=(g_1,\dots,g_m)$. 
Set 
\begin{equation}\label{eq:def.w.KKT.plus}
w:=\frac{1}{2}\times b(n+m,d+1,2m+n+1)+d\times {c(n+m,d+2,n+m)}\,.
\end{equation}
Then there exists $q\in P_w(g)[x, \lambda]$ such that $f-q$ vanishes on $V(h_{\KKT}^+)$, where $\lambda:=(\lambda_1,\dots,\lambda_m)$ and $h_{\KKT}$ is defined as in \eqref{eq:.polyKKT.plus}.
\end{theorem} 
\begin{proof}
Using Lemma \ref{lem:num.connected}, we decompose $V(h_{\KKT}^+)$ into semi-algebraically path connected components:
$Z_1,\dots,Z_s$ with  
\begin{equation}\label{eq:bound.on.s.KKT.plus}
s\le c(n+m,d+2,n+m)\,,
\end{equation}
since each entry of $h_{\KKT}^+$ has degree at most $d+2\ge 2$.
Accordingly Lemma \ref{lem:constant.KKT.plus} shows that $f$ is constant on each $Z_i$.
Thus $f(V(h_{\KKT}^+))$ is finite.
Set $r=|f(V(h_{\KKT}^+))|$.
From \eqref{eq:bound.on.s.KKT}, we get 
\begin{equation}\label{eq:ineq.KKT.plus}
r\le s\le c(n+m,d+2,n+m)\,.
\end{equation}
Set 
\begin{equation}
u:=b(n+m,d+1,2m+n+1)/2\,.
\end{equation}
By using Lemma \ref{lem:quadra}, there exists $q\in P_\xi(g)[x,\lambda]$ with $\xi=dr+u$ such that $f - q$ vanishes on $V(h_{\KKT}^+)$.
By \eqref{eq:ineq.KKT.plus} and \eqref{eq:def.w.KKT.plus}, $\xi\le w$, and hence $q\in P_w(g)[x,\lambda]$. 
\end{proof}
We apply Theorem \ref{theo:rep.KKT.plus} for polynomial optimization as follows:
\begin{theorem}\label{theo:pop.KKT.plus}
Let $f,g_1,\dots,g_m\in\R_d[x]$. 
Let $f^\star$ be as in problem \eqref{eq:pop} with $g=(g_1,\dots,g_m)$.
Assume that problem \eqref{eq:pop} has a global minimizer at which the Karush--Kuhn--Tucker conditions hold for this problem.
Let $h_{\KKT}^+$ be as in \eqref{eq:.polyKKT.plus} and let $w$ be as in \eqref{eq:def.w.KKT.plus}.
Set 
\begin{equation}
r=b(n+m+1,2w,m+n+2)/2\,.
\end{equation}
Then $\rho_r(f,\Pi g,h_{\KKT}^+)=f^\star$, where $\Pi g$ is defined as in \eqref{eq:prod.g}.
\end{theorem}
\begin{proof}
By assumption, there exists $(x^\star,\lambda^\star)\in V(h_{\KKT}^+)$ such that $x^\star$ is a global minimizer of \eqref{eq:pop}.
Since $S(g)=S(\Pi g)$, it implies that
\begin{equation}
\begin{array}{rl}
f^\star:=\min\limits_{x,\bar\lambda}& f(x)\\
\text{s.t.}& x\in S(\Pi g)\,,\,(x,\lambda)\in V(h_{\KKT}^+)\,,
\end{array}
\end{equation}
By assumption, Theorem \ref{theo:rep.KKT.plus} yields that there exists $q\in P_w(g)[x,\lambda]=Q_w(\Pi g)[x, \lambda]$ such that $f-f^\star-q$ vanishes on $V(h_{\KKT}^+)$.
Applying the sixth statement of Lemma \ref{lem:mom.sos} (by replacing $g$ with $\Pi g$), we obtain the conclusion.
\end{proof}
\section{Implicit degree bounds for the representations involving quadratic modules}
\label{sec:implicit.quadraic}
Let $\|\cdot\|$ denote the max norm of a polynomial on $[-1,1]^n$.
We recall in the following lemma the degree bound for Putinar's Positivstellensatz by Baldi and Mourain in \cite{baldi2021moment}:
\begin{lemma}\label{lem:Pu}
Let $f,g_1,\dots,g_m\in\R[x]$.
Assume that $f$ is positive on $S(g)$ with $g:=(g_1,\dots,g_m)$ and the following conditions:
\begin{enumerate}
\item $1-x_1^2-\dots-x_n^2\in Q_d(g)[x]$ for some $d\in\N$;
\item $\|g_i\|\le \frac{1}{2}\,,\,i=1,\dots,m$.
\end{enumerate} 
Let $f^\star$ be as in \eqref{eq:pop}.
Then there exist positive reals $\gamma(n,g)$ and $L(n,g)$ depending only on $n$ and $g$ such that $f \in Q_r(g)[x]$ if 
\begin{equation}\label{eq:bound.Putinar}
r\ge \gamma(n,g)\deg(f)^{3.5nL(n,g)}\left(\frac{\|f\|}{f^\star}\right)^{2.5nL(n,g)}\,.
\end{equation}
\end{lemma}
We denote by $v(n,f,g)$ the ceiling of the right hand side of \eqref{eq:bound.Putinar}.
Given $p\in\R[x]$ with $\xi=\lceil \deg(p)/2\rceil$, let $\tilde p:=x_0^{2\xi}p(\frac{x}{x_0\sqrt{2}}) \in\R[\bar x]$, where $\bar x:=(x_0,x)$.
The following lemma is a consequence of Lemmas \ref{lem:pos} and \ref{lem:Pu}:
\begin{lemma}\label{lem:rep-1.Pu}
Let $g_1,\dots,g_m\in\R_d[x]$.
Assume that the following conditions hold:
\begin{enumerate}
\item $\|\tilde g_i\|\le \frac{1}{2}\,,\,i=1,\dots,m$, $g_m:=\frac{1}{2}-x_1^2-\dots-x_n^2$;
\item $S(g)=\emptyset$ with $g:=(g_1,\dots,g_m)$. 
\end{enumerate}
Then there exists a positive real $u$ depending on $n$ and $g$ such that $-1 \in Q_{u}(g)[x]$.
\end{lemma}
Denote by $u(n,d,g)$ the parameter $u$ in Lemma \ref{lem:rep-1.Pu}.
\begin{proof}
Set $\eta=b(n,d,m+1)/2$.
Since $S(g)=\emptyset$, Lemma \ref{lem:pos} yields that there exists $\sigma_\alpha\in \Sigma^2[x]$ such that 
$\deg(\sigma_\alpha g^\alpha)\le 2\eta$ and
\begin{equation}\label{eq:rep.-1}
    -1=\sum_{\alpha\in\{0,1\}^n}\sigma_\alpha g^\alpha\,.
\end{equation}
We have $\tilde g_m=\frac{1}{2}x_0^2-x_1^2-\dots-x_n^2$.
Let $\eta_j=\lceil\deg(g_j)/2\rceil$.
From \eqref{eq:rep.-1}, we get
\begin{equation}\label{eq:equi}
    -x_0^{2\eta}=\sum_{\alpha\in\{0,1\}^n}\psi_\alpha \tilde g^\alpha\,,
\end{equation}
where $\psi_\alpha=x_0^{2(\eta-\eta_j)}\sigma_\alpha(\frac{x}{x_0\sqrt{2}})\in\Sigma^2[\bar x]$ and $\tilde g=(\tilde g_1,\dots,\tilde g_m)$.
Denote by $w\in\R[\bar x]$ the polynomial on the right hand side of \eqref{eq:equi}.
Then $w$ is non-negative on $S(\tilde g,\frac{1}{2}-x_0^2)$.
Since $0\in S(\tilde g,\frac{1}{2}-x_0^2)$, we get $S(\tilde g,\frac{1}{2}-x_0^2)\ne \emptyset$. 
On the other hand we have
\begin{equation}
    1-x_0^2-\dots-x_n^2=(\frac{1}{2}-x_0^2)+\tilde g_m\in Q(\tilde g,\frac{1}{2}-x_0^2)[\bar x]\,.
\end{equation}
Applying Lemma \ref{lem:Pu}, we obtain $u=v(n,w+\frac{1}{2^{\eta+1}},(\tilde g,\frac{1}{2}-x_0^2))$ such that
\begin{equation}
    -x_0^{2\eta}+\frac{1}{2^{\eta+1}}=w+\frac{1}{2^{\eta+1}}\in Q_u(\tilde g,\frac{1}{2}-x_0^2)[\bar x]\,.
\end{equation}
Letting $x_0=\frac{1}{\sqrt{2}}$ implies that $-\frac{1}{2^{\eta+1}}\in Q_u(g)[x]$, yielding the result.
\end{proof}
The following lemma is a direct consequence of Lemma \ref{lem:rep-1.Pu}:
\begin{lemma}\label{lem:rep-1.Pu.eq}
Let $g_1,\dots,g_m\in\R_d[x]$ and $h_1,\dots,h_l\in\R_{d+1}[x]$.
Assume that the following conditions hold:
\begin{enumerate}
\item $\|\tilde g_i\|\le \frac{1}{2}\,,\,i=1,\dots,m$, $g_m:=\frac{1}{2}-x_1^2-\dots-x_n^2$;
\item $S(g)\cap V(h)=\emptyset$ with $g:=(g_1,\dots,g_m)$ and $h:=(h_1,\dots,h_l)$. 
\end{enumerate} 
Set $h_{\max}:=\max\limits_{j=1,\dots,l}\|\tilde h_j\|$.
Then there exists a positive real 
\begin{equation}
w=u(n,d+1,(g,\frac{h}{2h_{\max}},-\frac{h}{2h_{\max}}))
\end{equation}
depending on $n$ and $g$ such that $-1 \in Q_{w}(g)[x]$.
\end{lemma}
Denote by $w(n,d,g,h)$ the parameter $w$ in Lemma \ref{lem:rep-1.Pu.eq}.

\subsection{The case without denominators}
\begin{lemma}\label{lem:quadra.module}
Let $f,g_1,\dots,g_m\in\R_d[x]$ and $h_1,\dots,h_l\in\R_{d+1}[x]$. 
Assume that the following conditions hold:
\begin{enumerate}
\item $\|\tilde g_i\|\le \frac{1}{2}$, $i=1,\dots,m$, $g_m:=\frac{1}{2}-x_1^2-\dots-x_n^2$;
\item $f$ is non-negative on $S(g)$ with $g=(g_1,\dots,g_m)$;
\item $f(V(h))$ is finite with $h=(h_1,\dots,h_l)$.
\end{enumerate}
  
Set $r:=|f(V(h))|$ and 
\begin{equation}
u:=\max_{t\in f(V(h))}w(n,d,g,(h,f-t))\,.
\end{equation}
Then there exist $q\in Q_t(g)[x,\bar\lambda]$ with $t=dr+u$ such that $f - q$ vanishes on $V(h)$.
\end{lemma}
Denote by $t(n,d,r,f,g,h)$ the parameter $t$ in Lemma \ref{lem:quadra.module}.
Note that it holds that 
\begin{equation}\label{eq:increasing.t}
t(n,d,r,f,g,h)\le t(n,d,r',f,g,h)\text{ if }r\le r'\,.
\end{equation}
\begin{proof}
By assumption, we get
$f(V(h)) = \{t_1 ,\dots, t_r \} \subset \R$,
where $t_i\ne t_j$ if $i\ne j$.
For $j=1,\dots,r$, let
$W_j:=V(h,f-t_j)$.
Then $W_j$ is a real variety generated by $l+1$ polynomials in $\R_{d+1}[x]$.
It is clear that $f(W_j)=\{t_j\}$.

Let $p_j\in\R[x]$, $j=1,\dots,r$, as in \eqref{eq:lagrange.pol}.
It is easy to check that $p_j(W_i)=\{\delta_{ji}\}$ and $\deg(p_j)\le d(r-1)$.
Without loss of generality, we assume that there is $s\in \{0,1,\dots,r-1\}$ such that $W_j\cap S(g)= \emptyset$, for $j=1,\dots,s$, and $W_i\cap S(g)\ne \emptyset$, for $i=s+1,\dots,r$.

Let $j\in\{1,\dots,s\}$.
Since $W_j \cap S(g) = \emptyset$, Theorem \ref{lem:rep-1.Pu} says that $-1 \in Q_u(g)[x]+I_u(h,f-t_j)[x]$.
It implies that there exists $v_j \in Q_u(g)[x]$ such that
$-1 = v_j$ on $W_j$. 
We have $f = s_1 - s_2$ for the SOS polynomials $s_1 =(f+\frac{1}{2})^2$ and $s_2 = f^2+\frac{1}{4}$.
It implies that $f = s_1 + v_j s_2$ on $W_j$.
Let $q_j = s_1 +v_j s_2 \in Q_{u+d}(g)[x,\bar\lambda]$.

Since $f\ge 0$ on $S(g)$, it holds that $f=t_i\ge 0$ on $W_i$, for $i=s+1,\dots,r$.  
Now letting $q$ as in \eqref{eq:rep.q}, we obtain $q\in Q_t(g)[x]$, and hence $f - q$ vanishes on $V(h)=W_1\cup\dots\cup W_r$.
\end{proof}
The implicit degree bound for the representation without denominators associated with quadratic module in \cite[Theorem 1]{mai2022exact} is stated in the following theorem:
\begin{theorem}\label{theo:rep.quadra.module}
Let $f,g_1,\dots,g_m\in\R_d[x]$ with $d\ge 1$. 
Assume that the following conditions hold:
\begin{enumerate}
\item $\|\tilde g_i\|\le \frac{1}{2}$, $i=1,\dots,m$ and $g_m:=\frac{1}{2}-x_1^2-\dots-x_n^2$;
\item $f$ is non-negative on $S(g)$  with $g:=(g_1,\dots,g_m)$ and $f(C(g))$ is finite.
\end{enumerate}
Set 
\begin{equation}\label{eq:def.omega}
\omega:=t(n,d,c(n+m+1,d+1,n+m+1),f,g,h_{\FJ})\,.
\end{equation}
Then there exists $q\in Q_\omega(g)[x,\bar \lambda]$ such that $f-q$ vanishes on $V(h_{\FJ})$, where $\bar\lambda:=(\lambda_0,\dots,\lambda_m)$ and $h_{\FJ}$ is defined as in \eqref{eq:.polyFJ}.
\end{theorem}
\begin{proof}
Using Lemma \ref{lem:num.connected}, we decompose $V(h_{\FJ})$ into semi-algebraically path connected components:
$Z_1,\dots,Z_s$ with $s$ satisfying \eqref{eq:bound.on.s}
(since each entry of $h_{\FJ}$ has degree at most $d+1\ge 2$).
Accordingly, \cite[Lemma 13]{mai2022exact} says that $f$ is constant on each $Z_i$.
Thus $f(V(h_{\FJ}))$ is finite.
Set $r=|f(V(h_{\FJ}))|$.
Then we get the inequality \eqref{eq:ineq}.
By using Lemma \ref{lem:quadra.module}, there exist $q\in Q_\xi(g)[x,\bar\lambda]$ with $\xi=t(n,d,r,f,g,h_{\FJ})$ such that $f - q$ vanishes on $V(h_{\FJ})$.
By \eqref{eq:ineq}, \eqref{eq:increasing.t} and \eqref{eq:def.omega}, $\xi\le \omega$, and hence $q\in Q_\omega(g)[x,\bar\lambda]$.
This completes the proof. 
\end{proof}
We apply Theorem \ref{theo:rep.quadra.module} for polynomial optimization as follows:
\begin{theorem}\label{theo:pop.quadra}
Let $f,g_1,\dots,g_m\in\R_d[x]$ with $d\ge 1$. 
Let $f^\star$ be as in problem \eqref{eq:pop} with $g=(g_1,\dots,g_m)$.
Assume that the following conditions hold:
\begin{enumerate}
\item $\|\tilde g_i\|\le \frac{1}{2}$, $i=1,\dots,m$ and $g_m:=\frac{1}{2}-x_1^2-\dots-x_n^2$;
\item problem \eqref{eq:pop} has a global minimizer and $f(C(g))$ is finite.
\end{enumerate}
Let $h_{\FJ}$ be as in \eqref{eq:.polyFJ} and let $\omega$ be as in \eqref{eq:def.omega}.
Set 
\begin{equation}
r=b(n+m+1,2\omega,m+n+2)/2\,.
\end{equation}
Then $\rho_r(f,g,h_{\FJ})=f^\star$.
\end{theorem}
\begin{proof}
Note that \cite[Lemma 17]{mai2022exact} implies \eqref{eq:equi.prob1}.
By assumption, Theorem \ref{theo:rep.quadra.module} yields that there exists $q\in Q_\omega(g)[x,\bar \lambda]$ such that $f-f^\star-q$ vanishes on $V(h_{\FJ})$.
Applying the sixth statement of Lemma \ref{lem:mom.sos}, we obtain the conclusion.
\end{proof}
\begin{remark}
The degree bound $\omega$ in Theorem \ref{theo:rep.quadra.module} is not very interesting,  and so is the rate $r$ in Theorem \ref{theo:pop.quadra}. 
It is because the bound $\omega$ in Theorem \ref{theo:rep.quadra.module} depends on all information of $f,g_j$.
Contrary to this, the degree bound $w$ in Theorem \ref{theo:rep} only depends on $n$  (the number of variables), $m$ (the number of polynomial inequalities $g_j$), and $d$ (the upper bound on the degree of $f,g_j$).
We conclude similarly to other representations without denominators involving quadratic modules in \cite{mai2022exact}.
\end{remark}
\subsection{The case with denominators}
\begin{lemma}\label{lem:quadra.module.deno}
Let $f,g_1,\dots,g_m\in\R_d[x]$ and $h_1,\dots,h_l\in\R_{d+1}[x]$. 
Assume that the following conditions hold:
\begin{enumerate}
\item $\|\tilde g_i\|\le \frac{1}{2}$, $i=1,\dots,m$, $g_m:=\frac{1}{2}-x_1^2-\dots-x_n^2$;
\item $f$ is non-negative on $S(g)$ with $g=(g_1,\dots,g_m)$;
\item $f(V(h)\backslash A)$ is finite with $h=(h_1,\dots,h_l)$ and $A\subset \R^n$.
\end{enumerate}
  
Set $r:=|f(V(h)\backslash A)|$ and 
\begin{equation}
u:=\max_{t\in f(V(h)\backslash A)}w(n,d,g,(h,f-t))\,.
\end{equation}
Then there exist $q\in Q_t(g)[x,\bar\lambda]$ with $t=dr+u$ such that $f - q$ vanishes on $V(h)\backslash A$.
\end{lemma}
Denote by $t'(n,d,r,f,g,h)$ the parameter $t$ in Lemma \ref{lem:quadra.module.deno}.
Note that it holds that 
\begin{equation}\label{eq:increasing.t.deno}
t'(n,d,r,f,g,h)\le t'(n,d,r',f,g,h)\text{ if }r\le r'\,.
\end{equation}
\begin{proof}
By assumption, we get
$f(V(h)\backslash A) = \{t_1 ,\dots, t_r \} \subset \R$,
where $t_i\ne t_j$ if $i\ne j$.
We now handle in much the same way as the proof of Lemma \ref{lem:quadra.module} to get $q\in Q_t(g)[x]$, and hence $f - q$ vanishes on $V(h)\backslash A\subset W_1\cup\dots\cup W_r$.
\end{proof}

In the following theorem, we state the implicit degree bound for the representation with denominators associated with quadratic modules in \cite[Theorem 9]{mai2022exact}:
\begin{theorem}\label{theo:rep.quadra.module.deno}
Let $f,g_1,\dots,g_m\in\R_d[x]$ with $d\ge 1$. 
Assume that the following conditions hold:
\begin{enumerate}
\item $\|\tilde g_i\|\le \frac{1}{2}$, $i=1,\dots,m$ and $g_m:=\frac{1}{2}-x_1^2-\dots-x_n^2$;
\item $f$ is non-negative on $S(g)$  with $g:=(g_1,\dots,g_m)$.
\end{enumerate}
Set 
\begin{equation}\label{eq:def.omega.deno}
\omega:=t'(n,d,2\times c(n+m+1,d+1,n+m+1),f,g,h_{\FJ})\,.
\end{equation}
Then there exists $q\in Q_\omega(g)[x,\bar \lambda]$ such that $\lambda_0(f-q)$ vanishes on $V(h_{\FJ})$, where $\bar\lambda:=(\lambda_0,\dots,\lambda_m)$ and $h_{\FJ}$ is defined as in \eqref{eq:.polyFJ}.
\end{theorem}
\begin{proof}
Using Lemma \ref{lem:num.connected}, we decompose $V(h_{\FJ})\backslash\{\lambda_0=0\}$ into semi-algebraically path connected components:
$Z_1,\dots,Z_s$ with $s$ satisfying \eqref{eq:bound.on.s.deno}
(since each entry of $h_{\FJ}$ has degree at most $d+1\ge 2$).
Accordingly, \cite[Lemma 24]{mai2022exact} says that $f$ is constant on each $Z_i$.
Thus $f(V(h_{\FJ})\backslash\{\lambda_0=0\})$ is finite.
Set $r=|f(V(h_{\FJ})\backslash\{\lambda_0=0\})|$.
Then we get the inequality \eqref{eq:ineq.deno}.
By using Lemma \ref{lem:quadra.module.deno}, there exist $q\in Q_\xi(g)[x,\bar\lambda]$ with $\xi=t'(n,d,r,f,g,h_{\FJ})$ such that $f - q$ vanishes on $V(h_{\FJ})\backslash \{\lambda_0=0\}$.
It implies that $\lambda_0(f - q)$ vanishes on $V(h_{\FJ})$.
By \eqref{eq:ineq.deno}, \eqref{eq:increasing.t.deno} and \eqref{eq:def.omega.deno}, $\xi\le \omega$, and hence $q\in Q_\omega(g)[x,\bar\lambda]$.
This completes the proof. 
\end{proof}
We apply Theorem \ref{theo:rep.quadra.module.deno} for polynomial optimization as follows:
\begin{theorem}\label{theo:pop.quadra.deno}
Let $f,g_1,\dots,g_m\in\R_d[x]$ with $d\ge 1$. 
Let $f^\star$ be as in problem \eqref{eq:pop} with $g=(g_1,\dots,g_m)$.
Assume that the following conditions hold:
\begin{enumerate}
\item $\|\tilde g_i\|\le \frac{1}{2}$, $i=1,\dots,m$ and $g_m:=\frac{1}{2}-x_1^2-\dots-x_n^2$;
\item problem \eqref{eq:pop} has a global minimizer $x^\star$.
\end{enumerate}
Let $h_{\FJ}$ be as in \eqref{eq:.polyFJ} and let $\omega$ be as in \eqref{eq:def.omega.deno}.
Assume that one of the following two conditions holds:
\begin{enumerate}
\item  the Karush--Kuhn--Tucker conditions hold for problem \eqref{eq:pop} at $x^\star$;
\item there exists a sequence of points $(x^{(t)},\bar \lambda^{(t)})_{t\in\N}$ in $(S(g)\times \R^{m+1})\cap V(h_{\FJ})$ such that $(f(x^{(t)}))_{t\in\N}$ converges to $\bar f^\star$ and  $\lambda_0^{(t)}>0$,
\end{enumerate}
Set 
\begin{equation}
r=b(n+m+1,2(\omega+1),m+n+2)/2\,.
\end{equation}
Then $\rho_r(f,g,h_{\FJ},\lambda_0)=f^\star$.
\end{theorem}
\begin{proof}
Note that \cite[Lemma 17]{mai2022exact} implies \eqref{eq:equi.prob1}.
Note that the first condition implies that the Fritz John conditions hold for problem \eqref{eq:pop} at $x^\star$ with multipliers $\bar\lambda^\star$ such that $\lambda_0^\star>0$, so that $(x^\star,\bar\lambda^\star)$ is a global minimizer for problem \eqref{eq:equi.prob1} with $\lambda_0^\star>0$.
The second condition implies that each $(x^{(t)},\bar \lambda^{(t)})$ is a feasible solution for problem \eqref{eq:equi.prob1} such that $(f(x^{(t)}))_{t\in\N}$ converges to $\bar f^\star$ and  $\lambda_0^{(t)}>0$.
The third statement of Lemma \ref{lem:mom.sos.deno} says that $\rho_k(f,g,h_{\FJ},\lambda_0)\le f^\star$, for every $k\in\N$.
In addition, Theorem \ref{theo:rep.quadra.module.deno} yields that there exists $q\in Q_\omega(\Pi g)[x,\bar \lambda]$ such that $\lambda_0(f-f^\star-q)$ vanishes on $V(h_{\FJ})$.
Applying the final statement of Lemma \ref{lem:mom.sos.deno} (by replacing $h,\theta$ with $h_{\FJ},\lambda_0$), we obtain the conclusion.
\end{proof}

\begin{remark}
Although they need the same assumption that the Karush--Kuhn--Tucker conditions hold at some global minimizer, the semidefinite relaxations in Theorem \ref{theo:pop.quadra.deno} have better complexity than the ones in Theorem \ref{theo:pop.KKT.plus}. 
It is because it does not require an exponential number of matrix variables in each semidefinite relaxation.
\end{remark}

\paragraph{Acknowledgements.} The author was supported by the funding from ANITI.

\begin{thebibliography}{10}

\bibitem{artin1927zerlegung}
E.~Artin.
\newblock {\"U}ber die zerlegung definiter funktionen in quadrate.
\newblock {\em Abhandlungen aus dem mathematischen Seminar der Universit{\"a}t
  Hamburg}, 5(1):100--115, 1927.

\bibitem{baldi2021moment}
L.~Baldi and B.~Mourrain.
\newblock {On Moment Approximation and the Effective Putinar's
  Positivstellensatz}.

\bibitem{boyd2004convex}
S.~Boyd and L.~Vandenberghe.
\newblock {\em Convex optimization}.
\newblock Cambridge university press, 2004.

\bibitem{coste2000introduction}
M.~Coste.
\newblock An introduction to semialgebraic geometry, 2000.

\bibitem{demmel2007representations}
J.~Demmel, J.~Nie, and V.~Powers.
\newblock {Representations of positive polynomials on noncompact semialgebraic
  sets via KKT ideals}.
\newblock {\em Journal of pure and applied algebra}, 209(1):189--200, 2007.

\bibitem{fang2020sum}
K.~Fang and H.~Fawzi.
\newblock The sum-of-squares hierarchy on the sphere and applications in
  quantum information theory.
\newblock {\em Mathematical Programming}, pages 1--30, 2020.

\bibitem{freund2016optimality}
R.~M. Freund.
\newblock Optimality conditions for constrained optimization problems.
\newblock {\em Massachusetts Institute of Technology (available at:
  \url{http://s3.amazonaws.com/mitsloan-php/wp-faculty/sites/30/2016/12/15031343/Optimality-Conditions-for-Constrained.pdf})},
  2016.

\bibitem{krivine1964anneaux}
J.-L. Krivine.
\newblock Anneaux pr{\'e}ordonn{\'e}s.
\newblock {\em Journal d'analyse math{\'e}matique}, 12(1):307--326, 1964.

\bibitem{lasserre2001global}
J.~B. Lasserre.
\newblock Global optimization with polynomials and the problem of moments.
\newblock {\em SIAM Journal on optimization}, 11(3):796--817, 2001.

\bibitem{laurent2021effective}
M.~Laurent and L.~Slot.
\newblock {An effective version of Schm\"udgen's Positivstellensatz for the
  hypercube}.
\newblock {\em arXiv preprint arXiv:2109.09528}, 2021.

\bibitem{lombardi2020elementary}
H.~Lombardi, D.~Perrucci, and M.-F. Roy.
\newblock {\em An elementary recursive bound for effective Positivstellensatz
  and Hilbert’s 17th problem}, volume 263.
\newblock American mathematical society, 2020.

\bibitem{mai2022exact}
N.~H.~A. Mai.
\newblock {Exact polynomial optimization strengthened with Fritz John
  conditions}.
\newblock {\em arXiv preprint arXiv:2205.04254}, 2022.

\bibitem{mai2021positivity}
N.~H.~A. Mai, J.-B. Lasserre, and V.~Magron.
\newblock Positivity certificates and polynomial optimization on non-compact
  semialgebraic sets.
\newblock {\em Mathematical Programming}, pages 1--43, 2021.

\bibitem{marshall2006representations}
M.~Marshall.
\newblock Representations of non-negative polynomials having finitely many
  zeros.
\newblock In {\em Annales de la Facult{\'e} des sciences de Toulouse:
  Math{\'e}matiques}, volume~15, pages 599--609, 2006.

\bibitem{marshall2009representations}
M.~Marshall.
\newblock Representations of non-negative polynomials, degree bounds and
  applications to optimization.
\newblock {\em Canadian Journal of Mathematics}, 61(1):205--221, 2009.

\bibitem{nie2013exact}
J.~Nie.
\newblock {An exact Jacobian SDP relaxation for polynomial optimization}.
\newblock {\em Mathematical Programming}, 137(1):225--255, 2013.

\bibitem{nie2014optimality}
J.~Nie.
\newblock {Optimality conditions and finite convergence of Lasserre's
  hierarchy}.
\newblock {\em Mathematical programming}, 146(1-2):97--121, 2014.

\bibitem{nie2006minimizing}
J.~Nie, J.~Demmel, and B.~Sturmfels.
\newblock Minimizing polynomials via sum of squares over the gradient ideal.
\newblock {\em Mathematical programming}, 106(3):587--606, 2006.

\bibitem{nie2007complexity}
J.~Nie and M.~Schweighofer.
\newblock {On the complexity of Putinar's Positivstellensatz}.
\newblock {\em Journal of Complexity}, 23(1):135--150, 2007.

\bibitem{pham2016genericity}
T.~S. Pham and H.~H. Vui.
\newblock {\em Genericity in polynomial optimization}, volume~3.
\newblock World Scientific, 2016.

\bibitem{putinar1993positive}
M.~Putinar.
\newblock Positive polynomials on compact semi-algebraic sets.
\newblock {\em Indiana University Mathematics Journal}, 42(3):969--984, 1993.

\bibitem{scheiderer2000sums}
C.~Scheiderer.
\newblock Sums of squares of regular functions on real algebraic varieties.
\newblock {\em Transactions of the American Mathematical Society},
  352(3):1039--1069, 2000.

\bibitem{scheiderer2003sums}
C.~Scheiderer.
\newblock Sums of squares on real algebraic curves.
\newblock {\em Mathematische zeitschrift}, 245(4):725--760, 2003.

\bibitem{scheiderer2006sums}
C.~Scheiderer.
\newblock Sums of squares on real algebraic surfaces.
\newblock {\em manuscripta mathematica}, 119(4):395--410, 2006.

\bibitem{schmudgen1991thek}
K.~Schm{\"u}dgen.
\newblock {The K-moment problem for compact semi-algebraic sets}.
\newblock {\em Mathematische Annalen}, 289(1):203--206, 1991.

\bibitem{schweighofer2004complexity}
M.~Schweighofer.
\newblock {On the complexity of Schm{\"u}dgen's Positivstellensatz}.
\newblock {\em Journal of Complexity}, 20(4):529--543, 2004.

\bibitem{slot2021sum}
L.~Slot.
\newblock {Sum-of-squares hierarchies for polynomial optimization and the
  Christoffel-Darboux kernel}.
\newblock {\em arXiv preprint arXiv:2111.04610}, 2021.

\bibitem{stengle1974nullstellensatz}
G.~Stengle.
\newblock {A Nullstellensatz and a Positivstellensatz in semialgebraic
  geometry}.
\newblock {\em Mathematische Annalen}, 207(2):87--97, 1974.

\end{thebibliography}
\bibliographystyle{abbrv}

\end{document}